\documentclass[dvipdfmx]{article}
\usepackage{amsthm, amscd, amsmath, amsfonts, mathrsfs}
\newtheorem{theo}{Theorem}[section]
\newtheorem{proposition}[theo]{Proposition}
\newtheorem{lemma}[theo]{Lemma}

\newtheorem{rem}[theo]{Remark}
\newtheorem{definition}[theo]{Definition}

\input xy
\xyoption{all}

\begin{document}

\title{A Gerby Deformation of Complex Tori and the Homological Mirror Symmetry}

\author{Kazushi Kobayashi\footnote{Department of Mathematics, Graduate School of Science, Osaka University, Toyonaka, Osaka, 560-0043, Japan. E-mail : k-kobayashi@cr.math.sci.osaka-u.ac.jp. 2020 Mathematics Subject Classification : 14J33, 53C08, 53D37 (primary), 14F08, 53D18 (secondary). Keywords : torus, gerbe, homological mirror symmetry.}}

\date{}

\maketitle

\maketitle

\begin{abstract}
Let $(X,\check{X})$ be a mirror pair of a complex torus $X$ and its mirror partner $\check{X}$. This mirror pair is described as the trivial special Lagrangian torus fibrations $X\rightarrow B$ and $\check{X}\rightarrow B$ on the same base space $B$ by SYZ construction. Then, we can associate a holomorphic line bundle $E(s,\mathcal{L})\rightarrow X$ to a pair $(s,\mathcal{L})$ of a Lagrangian section $s$ of $\check{X}\rightarrow B$ and a unitary local system $\mathcal{L}$ along it. In this paper, we first construct the deformation $X_{\mathcal{G}}$ of $X$ by a certain flat gerbe $\mathcal{G}$ and its mirror partner $\check{X}_{\mathcal{G}}$ from the mirror pair $(X,\check{X})$, and discuss deformations of objects $E(s,\mathcal{L})$ and $(s,\mathcal{L})$ over the deformed mirror pair $(X_{\mathcal{G}},\check{X}_{\mathcal{G}})$.
\end{abstract}

\tableofcontents

\section{Introduction}
\subsection{Motivation}
Let $X^n$ be an $n$-dimensional complex torus and $\check{X}^n$ be its mirror partner. For this mirror pair $(X^n,\check{X}^n)$, the homological mirror symmetry \cite{Kon} states the existence of an equivalence
\begin{equation*}
D^b(Coh(X^n))\cong Tr(Fuk(\check{X}^n))
\end{equation*}
as triangulated categories, where $D^b(Coh(X^n))$ is the bounded derived category of coherent sheaves over $X^n$ and $Tr(Fuk(\check{X}^n))$ is the enhanced triangulated category of the Fukaya category $Fuk(\check{X}^n)$ over $\check{X}^n$ \cite{Fukaya category} in the sense of Bondal-Kapranov-Kontsevich construction \cite{bondal, Kon}. In particular, when we regard the mirror pair $(X^n,\check{X}^n)$ as the trivial special Lagrangian torus fibrations $X^n\rightarrow B$ and $\check{X}^n\rightarrow B$ on the same base space $B$ by SYZ construction \cite{SYZ}, we can associate a holomorphic line bundle $E(s,\mathcal{L})\rightarrow X^n$ with an integrable connection $\nabla_{(s,\mathcal{L})}$ (we sometimes denote this object by $(E(s,\mathcal{L}),\nabla_{(s,\mathcal{L})})$ for simplicity) to each pair $(s,\mathcal{L})$ of a Lagrangian section $s$ of $\check{X}^n\rightarrow B$ and a unitary local system $\mathcal{L}$ along it as objects in the respective categories.

In this paper, we first discuss the deformation $X_{\mathcal{G}}^n$ of $X^n$ by a certain flat gerbe $\mathcal{G}$ and the corresponding mirror dual deformation $\check{X}_{\mathcal{G}}^n$ of $\check{X}^n$. Moreover, we also consider deformations of objects in the respective categories associated to the deformation from $(X^n,\check{X}^n)$ to $(X_{\mathcal{G}}^n,\check{X}_{\mathcal{G}}^n)$. In this section, we denote the deformed objects obtained from $(E(s,\mathcal{L}),\nabla_{(s,\mathcal{L})})$ and $(s,\mathcal{L})$ by $(E(s^{\mathcal{G}},\mathcal{L}^{\mathcal{G}}),\nabla_{(s^{\mathcal{G}},\mathcal{L}^{\mathcal{G}})})$ and $(s^{\mathcal{G}},\mathcal{L}^{\mathcal{G}})$, respectively.

Unfortunately, in this paper, we can not consider ``all'' deformations $X_{\mathcal{G}}^n$ of $X^n$ by ``arbitrary'' gerbes $\mathcal{G}$, and we actually focus on certain flat gerbes only as such gerbes $\mathcal{G}$ here. Below, we explain the reason why we only focus on such cases. 

Let us consider the case of elliptic curves, i.e., take a mirror pair $(X^1,\check{X}^1)$ of elliptic curves. Although we only treat the case of holomorphic line bundles in this paper, we can actually associate a stable vector bundle $\mathcal{E}(r,d)\rightarrow X^1$ such that the rank of $\mathcal{E}(r,d)$ is $r$ and the degree of $\mathcal{E}(r,d)$ is $d$ (with an integrable connection) to each pair $(s_{\frac{d}{r}},\mathcal{L})$ of an affine Lagrangian multi section $s_{\frac{d}{r}}$ whose slope is $\frac{d}{r}$ of $\check{X}^1\rightarrow B$ and a unitary local system $\mathcal{L}$ along it, where $r\in \mathbb{N}$ and $d\in \mathbb{Z}$ are relatively prime. Then, as mentioned in \cite{fm}, we can consider the $SL(2;\mathbb{Z})$-action (up to shifts) over $D^b(Coh(X^1))$ and each pair $(r,d)\in \mathbb{N}\times \mathbb{Z}\subset \mathbb{Z}^2$ is transformed under this action according to the standard $SL(2;\mathbb{Z})$-action on $\mathbb{Z}^2$. On the other hand, as the counterpart in the symplectic geometry side, we can similarly discuss the $SL(2;\mathbb{Z})$-action on $\check{X}^1$ via the homological mirror symmetry for $(X^1,\check{X}^1)$ (see also section 6 in \cite{slaction}). More precisely, for a given element $g\in SL(2;\mathbb{Z})$, we can define a symplectomorphism 
\begin{equation*}
\varphi_g : \check{X}^1\stackrel{\sim}{\rightarrow} \check{X}^1
\end{equation*}
by 
\begin{equation*}
\varphi_g(\check{x}):=g\check{x},
\end{equation*}
where $\check{x}$ is the local coordinate system of $\check{X}^1\approx (\mathbb{R}/\mathbb{Z})^2$, and note that this $\varphi_g$ is an automorphism on $\check{X}^1$. This $\varphi_g$ further induces autoequivalences over $Fuk(\check{X}^1)$ and $Tr(Fuk(\check{X}^1))$ :
\begin{equation*}
(s_{\frac{d}{r}},\mathcal{L})\mapsto (\varphi_g(s_{\frac{d}{r}}),(\varphi_g^{-1})^*\mathcal{L}).
\end{equation*}
In particular, it is well known that $SL(2;\mathbb{Z})$ is generated by the two elements
\begin{equation*}
\left( \begin{array}{cc} 1 & 0 \\ 1 & 1 \end{array} \right), \ \left( \begin{array}{cc} 0 & 1 \\ -1 & 0 \end{array} \right).
\end{equation*}
More explicitly, the element $\left( \begin{array}{cc} 1 & 0 \\ 1 & 1 \end{array} \right)$ corresponds to the autoequivalence over $D^b(Coh(X^1))$ which is induced from the tensor product $\otimes \mathcal{E}(1,1)$ (of course $\left( \begin{array}{cc} 1 & 0 \\ d & 1 \end{array} \right)$ ($d\in \mathbb{Z}$) corresponds to the autoequivalence over $D^b(Coh(X^1))$ which is induced from the tensor product $\otimes \mathcal{E}(1,d)$), and the element $\left( \begin{array}{cc} 0 & 1 \\ -1 & 0 \end{array} \right)$ corresponds to Fourier-Mukai transform over $D^b(Coh(X^1))$ (the dual elliptic curve of $X^1$ is $X^1$ itself). 

In this paper, as a generalization of the map $\varphi_{g_d} : \check{X}^1\stackrel{\sim}{\rightarrow} \check{X}^1$ which is defined by an element $g_d:=\left( \begin{array}{cc} 1 & 0 \\ d & 1 \end{array} \right)\in SL(2;\mathbb{Z})$ ($d\in \mathbb{Z}$) to the higher dimensional case, in the symplectic geometry side, we consider the diffeomorphism $\varphi_{g_{\tau}}$ on $\check{X}^n$ which is defined by an element
\begin{equation*}
g_{\tau}:=\left( \begin{array}{cc} I_n & O \\ \tau & I_n \end{array} \right)\in SL(2n;\mathbb{Z}),
\end{equation*}
where $\tau\in M(n;\mathbb{Z})$ and $I_n$ is the identity matrix of order $n$. Note that this $\varphi_{g_{\tau}}$ no longer preserves the symplectic structure on $\check{X}^n$ in general, namely, the triangulated functor over $Tr(Fuk(\check{X}^n))$ induced from $\varphi_{g_{\tau}}$ does not necessarily become an autoequivalence over $Tr(Fuk(\check{X}^n))$. Then, we generally need to consider the deformation of $X^n$ by the gerbe $\mathcal{G}_{\tau}$ depending on $\tau\in M(n;\mathbb{Z})$, i.e., the gerby deformed $n$-dimensional complex torus $X_{\mathcal{G}_{\tau}}^n$ as the counterpart in the complex geometry side (this gerbe $\mathcal{G}_{\tau}$ is actually flat). This is the gerby deformation itself which is discussed in this paper.

\subsection{Related works}
One of the most typical examples of mirror pairs is a pair of tori, so there are a lot of studies of the homological mirror symmetry for them. Historically, first, the homological mirror symmetry in the case of elliptic curves was discussed by Polishchuk and Zaslow \cite{elliptic}. After that, Fukaya studied the homological mirror symmetry in the case of abelian varieties as a generalization of the work by Polishchuk and Zaslow to the higher dimensional case \cite{Fuk} (see also \cite{dg, A-inf, abouzaid}). We remark that deformed complex tori (gerby deformed complex tori and noncommutative complex tori) and their mirror partners are not treated in those papers.

In this paper, we treat not only (usual) complex tori and mirror dual complexified symplectic tori but also (certain) gerby deformed complex tori and their mirror partners. In particular, we describe them in the framework of generalized complex geometry \cite{hitchin, Gual}. When we regard a complex manifold as a generalized complex manifold, it is known that a gerby deformation of it is interpreted as a B-field transform of the corresponding generalized complex structure. This fact is explained in \cite{Gual}. Moreover, Oren Ben-Bassat discusses the homological mirror symmetry for generalized complex manifolds under ``adapted requirement'' (see \cite[Definition 4.2]{part1}) in \cite{part1, part2}\footnote{Note that the definition of the mirror transform used in \cite{part1, part2} differs from the definition employed in this paper.}. For example, the generalized complex structure induced from the complex structure over an $n$-dimensional complex torus whose period matrix is given by $\mathbf{i}Y$ ($\mathbf{i}=\sqrt{-1}$, $Y\in M(n;\mathbb{R})$ is positive definite) satisfies ``adapted requirement'', and the generalized complex structure induced from the canonical complex structure over a (general) complex manifold also satisfies ``adapted requirement''. In particular, relatively general B-field transforms of the generalized complex structure induced from the canonical complex structure over a complex manifold (B-field transforms treated in this paper are special cases of them) and the mirror dual counterparts are briefly explained in subsection 5.2 in \cite{part2} (under ``adapted requirement''). 

Examples of previous works focusing on abelian varieties are \cite{kap1, kap2} by Kapustin and Orlov (see also \cite{cal}). In \cite{kap1, kap2}, they study the effect of the B-field on $D^b(Coh(X^n))$ in the case that $X^n$ is an abelian variety (this corresponds to the gerby deformation by a certain gerbe) and properties of the bounded derived category of twisted sheaves over the deformed abelian variety from the viewpoint of superconformal vertex algebras (we do not assume that $X^n$ is an abelian variety and superconformal vertex algebras are not treated in this paper). Furthermore, they also mention the effect of B-fields on Fukaya categories over given symplectic manifolds and the statement of the homological mirror symmetry conjecture which should be discussed for mirror pairs modified by B-fields.

On the other hand, in this paper, we intensively discuss certain gerby deformations of complex tori and the homological mirror symmetry without ``adapted requirement''. Although our discussions are basically based on the works explained in the above, we especially investigate deformations of objects in the respective categories in more detail than the existing works (we study the correspondence between special Lagrangian submanifolds and their mirror dual objects, i.e., deformed Hermitian-Yang-Mills connections in the respective deformed categories).

As examples of other related works, Fourier-Mukai partners of gerby deformed complex tori are studied in \cite{ben, block}, and there, it is verified that they are described as noncommutative complex tori.

\subsection{Main results and the plan of this paper}
Precisely speaking, the gerbe $\mathcal{G}_{\tau}$ mentioned in subsection 1.1 is actually the trivial gerbe with the flat connection which is locally defined by using $\tau\in M(n;\mathbb{Z})$. We first express this gerbe $\mathcal{G}_{\tau}$ in the sense of the Hitchin-Chatterjee's work \cite{chat, h} (see also \cite{han}), and describe the deformation $X_{\mathcal{G}_{\tau}}^n$ of $X^n$ by $\mathcal{G}_{\tau}$ and its mirror partner $\check{X}_{\mathcal{G}_{\tau}}^n$ via generalized complex geometry. Next, we give the deformation of an object $(E(s,\mathcal{L}),\nabla_{(s,\mathcal{L})})$ as the object $(E(s^{\mathcal{G}_{\tau}},\mathcal{L}^{\mathcal{G}_{\tau}}),\nabla_{(s^{\mathcal{G}_{\tau}},\mathcal{L}^{\mathcal{G}_{\tau}})})$ consisting of the twisted holomorphic line bundle $E(s^{\mathcal{G}_{\tau}},\mathcal{L}^{\mathcal{G}_{\tau}})\rightarrow X_{\mathcal{G}_{\tau}}^n$ and the twisted integrable connection $\nabla_{(s^{\mathcal{G}_{\tau}},\mathcal{L}^{\mathcal{G}_{\tau}})}$ on $E(s^{\mathcal{G}_{\tau}},\mathcal{L}^{\mathcal{G}_{\tau}})$ in the complex geometry side. Similarly, we also give the deformation of a pair $(s,\mathcal{L})$ as the pair $(s^{\mathcal{G}_{\tau}},\mathcal{L}^{\mathcal{G}_{\tau}})$ of the Lagrangian section $s^{\mathcal{G}_{\tau}}$ of $\check{X}_{\mathcal{G}_{\tau}}^n\rightarrow B_{\mathcal{G}_{\tau}}$ and the unitary local system $\mathcal{L}^{\mathcal{G}_{\tau}}$ along it in the setting of the SYZ trivial special Lagrangian torus fibration $\check{X}_{\mathcal{G}_{\tau}}^n\rightarrow B_{\mathcal{G}_{\tau}}$ on a base space $B_{\mathcal{G}_{\tau}}$. In particular, we can actually check that these deformed objects $(E(s^{\mathcal{G}_{\tau}},\mathcal{L}^{\mathcal{G}_{\tau}}),\nabla_{(s^{\mathcal{G}_{\tau}},\mathcal{L}^{\mathcal{G}_{\tau}})})$ and $(s^{\mathcal{G}_{\tau}},\mathcal{L}^{\mathcal{G}_{\tau}})$ are mirror dual to each other (Proposition \ref{mainth}). This result implies that our deformations discussed in this paper seem to be natural from the viewpoint of the homological mirror symmetry for $(X_{\mathcal{G}_{\tau}}^n,\check{X}_{\mathcal{G}_{\tau}}^n)$.

On the other hand, the study of special Lagrangian submanifolds is one of the central topics in the symplectic geometry, and it is expected that mirror dual objects corresponding to special Lagrangian submanifolds in the complex geometry side are deformed Hermitian-Yang-Mills connections \cite{leung, m}. Here, we give a rough explanation of deformed Hermitian-Yang-Mills connections. Let $M$ be a K$\ddot{\mathrm{a}}$hler manifold and $L\rightarrow M$ be a holomorphic line bundle with a Hermitian metric $h$. For this holomorphic line bundle $L$, we can consider the Hermitian connection $\nabla_h$ on $L$ and its curvature form $\Omega_h$. In general, if the curvature form $\Omega_h$ satisfies the so-called deformed Hermitian-Yang-Mills equation (see the explanations of B-cycles in section 2 in \cite{leung}), then we call the Hermitian connection $\nabla_h$ a deformed Hermitian-Yang-Mills connection on $L$. In particular, it is conjectured the following : the existence of a Hermitian metric $h$ on $L$ satisfying the deformed Hermitian-Yang-Mills equation is equivalent to that $L$ is stable in the sense of Bridgeland as an object in $D^b(Coh(M))$ \cite{stab, collins}. Thus, deformed Hermitian-Yang-Mills connections and special Lagrangian submanifolds in a mirror partner $\check{M}$ of $M$ are important from the viewpoint of the study of the stability conditions over $D^b(Coh(M))$ and $Tr(Fuk(\check{M}))$. In this paper, we also prove that the twisted integrable connection $\nabla_{(s^{\mathcal{G}_{\tau}},\mathcal{L}^{\mathcal{G}_{\tau}})}$ defines a deformed Hermitian-Yang-Mills connection on the twisted holomorphic line bundle $E(s^{\mathcal{G}_{\tau}},\mathcal{L}^{\mathcal{G}_{\tau}})\rightarrow X_{\mathcal{G}_{\tau}}^n$ if and only if the Lagrangian submanifold $s^{\mathcal{G}_{\tau}}$ (this precisely means that the graph of $s^{\mathcal{G}_{\tau}}$ is a Lagrangian submanifold in $\check{X}_{\mathcal{G}_{\tau}}^n$) becomes a special Lagrangian submanifold in $\check{X}_{\mathcal{G}_{\tau}}^n$, and this result is given in Theorem \ref{dhymsplag}.

This paper is organized as follows. In section 2, we construct a mirror partner $\check{X}^n$ of a given $n$-dimensional complex torus $X^n$ via generalized complex geometry. In section 3, we first define the flat gerbe $\mathcal{G}_{\tau}$, and after that, we consider the deformation $X_{\mathcal{G}_{\tau}}^n$ of $X^n$ by $\mathcal{G}_{\tau}$ and its mirror partner $\check{X}_{\mathcal{G}_{\tau}}^n$. In section 4, we explain the relation between holomorphic line bundles with integrable connections $(E(s,\mathcal{L}),\nabla_{(s,\mathcal{L})})$ and pairs $(s,\mathcal{L})$ of Lagrangian sections $s$ of $\check{X}^n\rightarrow B$ with unitary local systems $\mathcal{L}$ along them. In section 5, we discuss deformations of objects $(E(s,\mathcal{L}),\nabla_{(s,\mathcal{L})})$ and $(s,\mathcal{L})$ over $(X_{\mathcal{G}_{\tau}}^n,\check{X}_{\mathcal{G}_{\tau}}^n)$. In particular, for given mirror dual objects $(E(s,\mathcal{L}),\nabla_{(s,\mathcal{L})})$ and $(s,\mathcal{L})$, we confirm that the deformations of them are also mirror dual to each other. This result is given in Proposition \ref{mainth}. In section 6, for deformed objects $(E(s^{\mathcal{G}_{\tau}},\mathcal{L}^{\mathcal{G}_{\tau}}),\nabla_{(s^{\mathcal{G}_{\tau}},\mathcal{L}^{\mathcal{G}_{\tau}})})$ and $(s^{\mathcal{G}_{\tau}},\mathcal{L}^{\mathcal{G}_{\tau}})$, we prove that $\nabla_{(s^{\mathcal{G}_{\tau}},\mathcal{L}^{\mathcal{G}_{\tau}})}$ defines a deformed Hermitian-Yang-Mills connection on $E(s^{\mathcal{G}_{\tau}},\mathcal{L}^{\mathcal{G}_{\tau}})$ if and only if the Lagrangian submanifold as the graph of $s^{\mathcal{G}_{\tau}}$ becomes a special Lagrangian submanifold in $\check{X}_{\mathcal{G}_{\tau}}^n$. This result is given in Theorem \ref{dhymsplag}.

\section{A mirror pair $(X^n,\check{X}^n)$}
Let $T$ be a complex matrix of order $n\in \mathbb{N}$ such that $\mathrm{Im}T$ is positive definite and $\mathrm{det}T\not=0$\footnote{Although we do not need to assume the condition $\mathrm{det}T\not=0$ when we generally define an $n$-dimensional complex torus $X^n=\mathbb{C}^n/\mathbb{Z}^n\oplus T\mathbb{Z}^n$, in our setting described below, the mirror partner $\check{X}^n$ of $X^n$ does not exist if $\mathrm{det}T=0$. However, we can avoid this problem and discuss the homological mirror symmetry even if $\mathrm{det}T=0$ by modifying the definition of the mirror partner of $X^n$ and a class of objects which we treat. This fact is discussed in \cite{kazushi}.}. For simplicity, we denote the $n$-dimensional complex torus $\mathbb{C}^n/\mathbb{Z}^n\oplus T\mathbb{Z}^n$ by $X^n$ :
\begin{equation*}
X^n:=\mathbb{C}^n/\mathbb{Z}^n\oplus T\mathbb{Z}^n.
\end{equation*}
In this section, we define a mirror partner $\check{X}^n$ of $X^n$ via the framework of generalized complex geometry (\cite{Gual, part1, part2} etc.). Throughout this section, for a smooth manifold $M$ and a smooth vector bundle $E$, $TM$, $T^*M$ and $\Gamma(E)$ denote the tangent bundle on $M$, the cotangent bundle on $M$ and the space of smooth sections of $E$, respectively.

First, let us prepare notations. We sometimes regard $X^n$ as a $2n$-dimensional real torus $\mathbb{R}^{2n}/\mathbb{Z}^{2n}$. We fix an $\epsilon>0$ small enough and let
\begin{align*}
O_{m_1\cdots m_n}^{l_1\cdots l_n}:=\biggl\{ \left( \begin{array}{ccc}x \\ y \end{array} \right)\in X^n \ | \ &\frac{l_j-1}{3}-\epsilon <x_j <\frac{l_j}{3}+\epsilon, \\
&\frac{m_k-1}{3}-\epsilon <y_k <\frac{m_k}{3}+\epsilon, \ j,k=1,\cdots, n \biggr\}
\end{align*}
be subsets in $X^n$, where $l_j$, $m_k=1,2,3$, $x:=(x_1,\cdots, x_n)^t$, $y:=(y_1,\cdots, y_n)^t$, and we identify $x_i\sim x_i+1$, $y_i\sim y_i+1$ for each $i=1,\cdots, n$. If necessary, we denote
\begin{equation*}
O_{m_1\cdots (m_k=m)\cdots m_n}^{l_1\cdots (l_j=l)\cdots l_n}
\end{equation*}
instead of $O_{m_1\cdots m_n}^{l_1\cdots l_n}$ in order to specify the values $l_j=l$, $m_k=m$. Then, $\{ O_{m_1\cdots m_n}^{l_1\cdots l_n} \}_{l_j, m_k=1,2,3}$ is an open cover of $X^n$, and we define the local coordinates of $O_{m_1\cdots m_n}^{l_1\cdots l_n}$ by $(x_1,\cdots, x_n,y_1,\cdots, y_n)^t\in \mathbb{R}^{2n}$. Moreover, we locally express the complex coordinates $z:=(z_1,\cdots, z_n)^t$ of $X^n$ as $z=x+Ty$.

In the above situation, we consider a linear map
\begin{equation*}
\mathcal{I}_J : \Gamma(TX^n\oplus T^*X^n)\rightarrow \Gamma(TX^n\oplus T^*X^n)
\end{equation*}
which is expressed locally as 
\begin{align}
&\mathcal{I}_J \left( \frac{\partial}{\partial x}^t, \frac{\partial}{\partial y}^t, dx^t, dy^t \right) \notag \\
&=\left( \frac{\partial}{\partial x}^t, \frac{\partial}{\partial y}^t, dx^t, dy^t \right) \notag \\
&\hspace{3.5mm} \left( \begin{array}{cccc} -XY^{-1} & -Y-XY^{-1}X & O & O \\ Y^{-1} & Y^{-1}X & O & O \\ O & O & (Y^{-1})^tX^t & -(Y^{-1})^t \\ O & O & Y^t+X^t(Y^{-1})^tX^t & -X^t(Y^{-1})^t \end{array} \right), \label{comp}
\end{align}
where
\begin{equation*}
X:=\mathrm{Re}T, \ Y:=\mathrm{Im}T,
\end{equation*}
and
\begin{align*}
&\frac{\partial}{\partial x}:=\left(\frac{\partial}{\partial x_1},\cdots, \frac{\partial}{\partial x_n}\right)^t, \ \frac{\partial}{\partial y}:=\left(\frac{\partial}{\partial y_1},\cdots, \frac{\partial}{\partial y_n}\right)^t, \\
&dx:=(dx_1,\cdots, dx_n)^t, \ dy:=(dy_1,\cdots, dy_n)^t.
\end{align*}
We can check that this linear map $\mathcal{I}_J$ defines a generalized complex structure over $X^n$.

Let us define a mirror partner $\check{X}^n$ of the $n$-dimensional complex torus $X^n$. In general, for a $2n$-dimensional real torus $T^{2n}$ equipped with a generalized complex structure $\mathcal{I}$ over $T^{2n}$, its mirror dual generalized complex structure $\check{\mathcal{I}}$ over $T^{2n}$ is defined by
\begin{equation}
\check{\mathcal{I}}:=\left( \begin{array}{cccc} I_n & O & O & O \\ O & O & O & I_n \\ O & O & I_n & O \\ O & I_n & O & O \end{array} \right) \mathcal{I} \left( \begin{array}{cccc} I_n & O & O & O \\ O & O & O & I_n \\ O & O & I_n & O \\ O & I_n & O & O \end{array} \right), \label{mirrordef}
\end{equation}
where $I_n$ is the identity matrix of order $n$ (cf. \cite{part1, part2, Kaj, kazushi}). According to the definition (\ref{mirrordef}), we can calculate the mirror dual generalized complex structure $\check{\mathcal{I}}_J$ of $\mathcal{I}_J$ locally as 
\begin{align}
&\check{\mathcal{I}}_J \left( \frac{\partial}{\partial \check{x}}^t, \frac{\partial}{\partial \check{y}}^t, d\check{x}^t, d\check{y}^t \right) \notag \\
&=\left( \frac{\partial}{\partial \check{x}}^t, \frac{\partial}{\partial \check{y}}^t, d\check{x}^t, d\check{y}^t \right) \notag \\
&\hspace{3.5mm} \left( \begin{array}{cccc} -XY^{-1} & O & O & -Y-XY^{-1}X \\ O & -X^t(Y^{-1})^t & Y^t+X^t(Y^{-1})^tX^t & O \\ O & -(Y^{-1})^t & (Y^{-1})^tX^t & O \\ Y^{-1} & O & O & Y^{-1}X \end{array} \right), \label{symp}
\end{align}
where $\check{x}:=(x^1,\cdots, x^n)^t$, $\check{y}:=(y^1,\cdots, y^n)^t$ are the local coordinates of a $2n$-dimensional real torus $\mathbb{R}^{2n}/\mathbb{Z}^{2n}\approx (\mathbb{R}^n/\mathbb{Z}^n)\times (\mathbb{R}^n/\mathbb{Z}^n)$, and 
\begin{align*}
&\frac{\partial}{\partial \check{x}}:=\left(\frac{\partial}{\partial x^1},\cdots, \frac{\partial}{\partial x^n}\right)^t, \ \frac{\partial}{\partial \check{y}}:=\left(\frac{\partial}{\partial y^1},\cdots, \frac{\partial}{\partial y^n}\right)^t, \\
&d\check{x}:=(dx^1,\cdots, dx^n)^t, \ d\check{y}:=(dy^1,\cdots, dy^n)^t.
\end{align*}
We can actually rewrite the representation matrix in (\ref{symp}) to
\begin{align*}
&\left( \begin{array}{cccc} I_n & O & O & O \\ O & I_n & O & O \\ O & -\mathrm{Re}(-(T^{-1})^t) & I_n & O \\ (\mathrm{Re}(-(T^{-1})^t))^t & O & O & I_n \end{array} \right) \\
&\left( \begin{array}{cccc} O & O & O & -((\mathrm{Im}(-(T^{-1})^t))^{-1})^t \\ O & O & (\mathrm{Im}(-(T^{-1})^t))^{-1} & O \\ O & -\mathrm{Im}(-(T^{-1})^t) & O & O \\ (\mathrm{Im}(-(T^{-1})^t))^t & O & O & O \end{array} \right) \\
&\left( \begin{array}{cccc} I_n & O & O & O \\ O & I_n & O & O \\ O & \mathrm{Re}(-(T^{-1})^t) & I_n & O \\ -(\mathrm{Re}(-(T^{-1})^t))^t & O & O & I_n \end{array} \right).
\end{align*}
In particular, the assumption $\mathrm{det}T\not=0$ implies $\mathrm{det}(Y+XY^{-1}X)\not=0$, and
\begin{align*}
&\mathrm{Re}(-(T^{-1})^t)=-((Y+XY^{-1}X)^{-1})^tX^t(Y^{-1})^t, \\ 
&\mathrm{Im}(-(T^{-1})^t)=((Y+XY^{-1}X)^{-1})^t.
\end{align*}
Thus, when we define a mirror partner of the $n$-dimensional complex torus $X^n$ according to the definition (\ref{mirrordef}), it is given by the $2n$-dimensional real torus $\mathbb{R}^{2n}/\mathbb{Z}^{2n}$ with the complexified symplectic form
\begin{equation*}
\tilde{\omega}^{\vee}:=2\pi \mathbf{i}d\check{x}^t (T^{-1})^t d\check{y}.
\end{equation*}
Here, $\mathbf{i}:=\sqrt{-1}$ and when we decompose $\tilde{\omega}^{\vee}$ into 
\begin{equation*}
\tilde{\omega}^{\vee}=2\pi d\check{x}^t \mathrm{Im}(-(T^{-1})^t) d\check{y}-\mathbf{i} \Bigl\{ 2\pi d\check{x}^t \mathrm{Re}(-(T^{-1})^t) d\check{y} \Bigr\},
\end{equation*}
its real part
\begin{equation}
\omega^{\vee}:=2\pi d\check{x}^t \mathrm{Im}(-(T^{-1})^t) d\check{y} \label{sympform}
\end{equation}
gives a symplectic form on $\mathbb{R}^{2n}/\mathbb{Z}^{2n}$, and the closed $2$-form
\begin{equation}
B^{\vee}:=2\pi d\check{x}^t \mathrm{Re}(-(T^{-1})^t) d\check{y} \label{bfield}
\end{equation}
is often called the B-field. Moreover, we define two matrices $\omega_{\rm mat}^{\vee}$ and $B_{\rm mat}^{\vee}$ by
\begin{equation*}
\omega_{\rm mat}^{\vee}:=\mathrm{Im}(-(T^{-1})^t)
\end{equation*}
and
\begin{equation*}
B_{\rm mat}^{\vee}:=\mathrm{Re}(-(T^{-1})^t),
\end{equation*}
respectively, namely,
\begin{equation*}
\omega^{\vee}=2\pi d\check{x}^t \omega_{\rm mat}^{\vee} d\check{y}, \ B^{\vee}=2\pi d\check{x}^t B_{\rm mat}^{\vee} d\check{y}.
\end{equation*}
We sometimes use these notations $\omega_{\rm mat}^{\vee}$ and $B_{\rm mat}^{\vee}$. Hereafter, for simplicity, we denote
\begin{equation*}
\check{X}^n:=\Bigl( \mathbb{R}^{2n}/\mathbb{Z}^{2n}, \ \tilde{\omega}^{\vee}=2\pi \mathbf{i}d\check{x}^t (T^{-1})^t d\check{y} \Bigr).
\end{equation*}

\section{A gerby deformation $X_{\mathcal{G}_{\tau}}^n$ of $X^n$ and its mirror partner $\check{X}_{\mathcal{G}_{\tau}}^n$}
In this section, we consider the deformation $X_{\mathcal{G}_{\tau}}^n$ of $X^n$ by a flat gerbe $\mathcal{G}_{\tau}$ ($\tau \in M(n;\mathbb{Z})$) and define its mirror partner $\check{X}_{\mathcal{G}_{\tau}}^n$. In particular, we also confirm that the deformation from $\check{X}^n$ to $\check{X}_{\mathcal{G}_{\tau}}^n$ is induced from a symplectomorphism $\varphi_{g_{\tau}} : \check{X}^n \stackrel{\sim}{\rightarrow} \check{X}_{\mathcal{G}_{\tau}}^n$ defined by an element $g_{\tau}\in SL(2n;\mathbb{Z})$ as mentioned in Introduction.

Before starting main discussions, we give a remark. Below, we only consider a period matrix $T$ and $\tau\in M(n;\mathbb{Z})$ which satisfy the condition $\tau T\not=(\tau T)^t$ when we discuss the deformation of $X^n=\mathbb{C}^n/\mathbb{Z}^n\oplus T\mathbb{Z}^n$. Frankly speaking, the relation $\tau T=(\tau T)^t$ holds if and only if the corresponding deformation preserves the (generalized) complex structure of $X^n$, so at least, we need not consider a (non-trivial) gerby deformation of $X^n$ in the case $\tau T=(\tau T)^t$. We further remark that the corresponding deformation is interpreted as the autoequivalence over $D^b(Coh(X^n))$ which is induced from the tensor product $\otimes L_{\tau}$ by the holomorphic line bundle $L_{\tau}\rightarrow X^n$ depending on $\tau \in M(n;\mathbb{Z})$ (the condition $\tau T=(\tau T)^t$ implies the holomorphicity of $L_{\tau}$) under the assumption $\tau T=(\tau T)^t$. Such a fact is mentioned in \cite{Gual} for instance. Therefore, we focus on the case $\tau T\not=(\tau T)^t$ only below.

We first prepare notations. For each open set $O_{m_1\cdots m_n}^{l_1\cdots l_n}$ in $X^n$, we put
\begin{equation*}
l:=(l_1\cdots l_n), \ m:=(m_1\cdots m_n),
\end{equation*}
and denote $O_{m_1\cdots m_n}^{l_1\cdots l_n}$ by $O_m^l$ for simplicity :
\begin{equation*}
O_m^l:=O_{m_1\cdots m_n}^{l_1\cdots l_n}.
\end{equation*}
Namely, when we set
\begin{equation*}
I:=\Bigl\{ (l;m)=(l_1\cdots l_n ; m_1\cdots m_n) \ | \ l_1,\cdots, l_n, m_1,\cdots, m_n=1,2,3 \Bigr\},
\end{equation*}
$\{ O_m^l \}_{(l;m)\in I}$ gives an open cover of $X^n$ :
\begin{equation*}
X^n=\bigcup_{(l;m)\in I}O_m^l.
\end{equation*}
Moreover, we denote
\begin{equation*}
O_{mm'}^{ll'}:=O_m^l \cap O_{m'}^{l'}, \ O_{mm'm''}^{ll'l''}:=O_m^l \cap O_{m'}^{l'} \cap O_{m''}^{l''}
\end{equation*}
for $(l;m)$, $(l';m')$, $(l'';m'')\in I$, and define a smooth complex line bundle $L_{mm'm''}^{ll'l''}\rightarrow O_{mm'm''}^{ll'l''}$ by
\begin{equation*}
L_{mm'm''}^{ll'l''}:=L_{mm'}^{ll'}|_{O_{mm'm''}^{ll'l''}}\otimes L_{m'm''}^{l'l''}|_{O_{mm'm''}^{ll'l''}}\otimes L_{m''m}^{l''l}|_{O_{mm'm''}^{ll'l''}}
\end{equation*}
for given three smooth complex line bundles $L_{mm'}^{ll'}\rightarrow O_{mm'}^{ll'}$, $L_{m'm''}^{l'l''}\rightarrow O_{m'm''}^{l'l''}$, $L_{m''m}^{l''l}\rightarrow O_{m''m}^{l''l}$.

Here, we recall the definition of gerbes\footnote{Gerbes were originally introduced by Giraud in \cite{giraud}.} in the sense of Hitchin-Chatterjee's work \cite{chat, h} briefly (see also \cite{han}). Roughly speaking, for a given smooth manifold $M$, a gerbe on $M$ is determined by an open cover $\{ U_i \}_{i\in I}$ of $M$ and an element of $H^2(M;\mathcal{A}^*)$, where $\mathcal{A}^*$ is the sheaf of nowhere vanishing smooth functions on $M$, so we need to focus on each intersection $U_{ijk}$ of three open sets $U_i$, $U_j$, $U_k$ if we consider gerbes according to the original definition (we use the notations $U_{ij}:=U_i\cap U_j$, $U_{ijk}:=U_i\cap U_j\cap U_k$, etc. in the above sense). On the other hand, Hitchin and Chatterjee propose a way of defining gerbes as families of line bundles defined on intersections $U_{ij}$ of two open sets. Precisely speaking, this definition is given as follows.
\begin{definition} \label{gerbe}
A gerbe $\mathcal{G}(I,L,\theta)$ on a smooth manifold $M$ is defined by the following data $:$ \\
$(\mathrm{i})$ An open cover $\{ U_i \}_{i\in I}$ of $M$. \\
$(\mathrm{ii})$ A family $L:=\{ L_{ij} \}_{i,j\in I}$ of smooth complex line bundles $L_{ij}\rightarrow U_{ij}$ such that $L_{ji}\cong L_{ij}^*$ on each $U_{ij}$ and $L_{ijk}:=L_{ij}|_{U_{ijk}}\otimes L_{jk}|_{U_{ijk}}\otimes L_{ki}|_{U_{ijk}}$ is isomorphic to $\mathcal{O}_{ijk}$ on each $U_{ijk}$, where $L_{ij}^*$ is the dual of $L_{ij}$ and $\mathcal{O}_{ijk}$ is the trivial complex line bundle on $U_{ijk}$. \\
$(\mathrm{iii})$ A family $\theta:=\{ \theta_{ijk} \}_{i,j,k\in I}$ of smooth sections $\theta_{ijk}\in \Gamma(U_{ijk};L_{ijk})$ such that $\theta_{ijk}=\theta_{jik}^{-1}=\theta_{ikj}^{-1}=\theta_{kji}^{-1}$ on each $U_{ijk}$ and 
\begin{equation*}
(\delta \theta)_{ijkl}=\theta_{jkl}\otimes \theta_{ikl}^{-1}\otimes \theta_{ijl}\otimes \theta_{ijk}^{-1}=1 
\end{equation*}
holds on each $U_{ijkl}$, where $\Gamma(U_{ijk};L_{ijk})$ denotes the space of smooth sections of $L_{ijk}$ on $U_{ijk}$ and $\delta$ is the derivative in the sense of $\check{C}$ech cohomology.
\end{definition}
\begin{rem}
In the condition $(\mathrm{ii})$ in Definition \ref{gerbe}, a trivialization $L_{ijk}\cong \mathcal{O}_{ijk}$ is not being fixed. Moreover, we see that
\begin{align*}
&(\delta^2 L)_{ijkl} \\
&=(L_{jk}\otimes L_{kl}\otimes L_{lj})\otimes(L_{ik}\otimes L_{kl}\otimes L_{li})^{-1}\otimes (L_{ij}\otimes L_{jl}\otimes L_{li})\otimes (L_{ij}\otimes L_{jk}\otimes L_{ki})^{-1} \\ &=\mathcal{O}_{ijkl}
\end{align*}
holds on each $U_{ijkl}$ since each smooth complex line bundle $L_{ij}$ satisfies the condition $L_{ji}\cong L_{ij}^*=L_{ij}^{-1}$, where $\mathcal{O}_{ijkl}$ is the trivial complex line bundle on $U_{ijkl}$. Hence, the condition $(\mathrm{iii})$ in Definition \ref{gerbe} states that $(\delta \theta)_{ijkl}$ is the canonical trivialization $1$ of $(\delta^2 L)_{ijkl}=\mathcal{O}_{ijkl}$.
\end{rem}
A connection over a gerbe $\mathcal{G}(I,L,\theta)$ is defined as follows.
\begin{definition} \label{0-conn}
Let $\mathcal{G}(I,L,\theta)$ be a gerbe on a smooth manifold $M$. We call a family $\nabla:=\{ \nabla_{ij} \}_{i,j\in I}$ of connections $\nabla_{ij}$ on $L_{ij}\rightarrow U_{ij}$ such that
\begin{equation*}
\nabla_{ijk}\theta_{ijk}=0
\end{equation*}
on each $U_{ijk}$ a 0-connection over $\mathcal{G}(I,L,\theta)$. Here, $\nabla_{ijk}$ denotes the connection on $L_{ijk}$ which is induced from the connections $\nabla_{ij}$, $\nabla_{jk}$, $\nabla_{ki}$ $:$
\begin{equation*}
\nabla_{ijk}:=\nabla_{ij}+\nabla_{jk}+\nabla_{ki}.
\end{equation*}
\end{definition}
In particular, note that each $\nabla_{ijk}$ in Definition \ref{0-conn} is a flat connection since $\nabla_{ijk}$ is a connection on $L_{ijk}$ which is isomorphic to $\mathcal{O}_{ijk}$. Therefore, the curvature form $\Omega_{ij}$ of each connection $\nabla_{ij}$ satisfies the relation 
\begin{equation*}
\Omega_{ij}+\Omega_{jk}+\Omega_{ki}=(\delta \Omega)_{ijk}=0,
\end{equation*}
so $\{ \Omega_{ij} \}_{i,j\in I}$ defines a representative in $H^1(M;\mathcal{A}^2)$, where $\mathcal{A}^2$ is the sheaf of smooth 2-forms on $M$. This implies that the existence of $\beta_i\in \Gamma(U_i;\mathcal{A}^2)$ such that
\begin{equation*}
(\delta \beta)_{ij}=\Omega_{ij}
\end{equation*}
for each $\Omega_{ij}$ since $\mathcal{A}^2$ is a fine sheaf, i.e., $H^1(M;\mathcal{A}^2)\cong 0$.
\begin{definition} \label{1-conn}
We call the family $\beta:=\{ \beta_i \}_{i\in I}$ of 2-forms $\beta_i$ in the above a 1-connection over $\mathcal{G}(I,L,\theta)$ compatible with the 0-connection $\nabla$.
\end{definition}
On the other hand, the definition of a ``holomorphic'' gerbe on a complex manifold is also given by replacing ``smooth complex line bundles $L_{ij}\rightarrow U_{ij}$'' and ``smooth sections $\theta_{ijk}$'' with ``holomorphic line bundles $L_{ij}\rightarrow U_{ij}$'' and ``holomorphic sections $\theta_{ijk}$'', respectively in Definition \ref{gerbe}. In particular, we can naturally rewrite the notions of 0 and 1-connections in Definition \ref{0-conn} and Definition \ref{1-conn} to the ``holomorphic'' version (see section 5 in \cite{chat}).

Now, let us consider the trivial gerbe $\mathcal{G}(I,\mathcal{O},\theta)$ on $X^n$ :
\begin{align*}
&\mathcal{O}:=\Bigl\{ \mathcal{O}_{mm'}^{ll'}\rightarrow O_{mm'}^{ll'} \Bigr\}_{(l;m),(l';m')\in I}, \\ &\theta:=\Bigl\{ \theta_{mm'm''}^{ll'l''}=1\in \Gamma(O_{mm'm''}^{ll'l''};\mathcal{O}_{mm'm''}^{ll'l''}) \Bigr\}_{(l;m),(l';m'),(l'';m'')\in I},
\end{align*}
and by using a matrix $\tau \in M(n;\mathbb{Z})$, we define a flat connection over $\mathcal{G}(I,\mathcal{O},\theta)$ as follows. For each $(l;m)$, $(l';m')\in I$, we take a 1-form $\omega_{mm'}^{ll'}$ on the open set $O_{mm'}^{ll'}$ which is expressed locally as
\begin{equation*}
\omega_{mm'}^{ll'}:=
\begin{cases}
2\pi (\tau_{1j}dy_1+\cdots \tau_{nj}dy_n) & (l_j=1, l_j'=3, l_k=l_k' (k\not=j), m=m') \\ 0 & (\mathrm{otherwise}),
\end{cases}
\end{equation*}
where $j=1,\cdots, n$. By using them, we locally define a flat connection $\nabla_{mm'}^{ll'}$ on the trivial line bundle $\mathcal{O}_{mm'}^{ll'}\rightarrow O_{mm'}^{ll'}$ by
\begin{equation*}
\nabla_{mm'}^{ll'}:=d-\mathbf{i}\omega_{mm'}^{ll'}
\end{equation*}
for each $(l;m)$, $(l';m')\in I$, where $d$ denotes the exterior derivative. Then, the family of $\nabla_{mm'}^{ll'}$ gives a 0-connection
\begin{equation*}
\nabla:=\Bigl\{ \nabla_{mm'}^{ll'} \Bigr\}_{(l;m),(l';m')\in I}
\end{equation*}
over $\mathcal{G}(I,\mathcal{O},\theta)$. Moreover, for each $(l;m)\in I$, we consider a 2-form $B_m^l$ on the open set $O_m^l$ which is expressed locally as
\begin{equation*}
B_m^l:=2\pi dx^t \tau ^t dy.
\end{equation*}
Note that this $B_m^l$ is actually defined globally. We can easily check that the family of $-\mathbf{i}B_m^l$ defines a 1-connection
\begin{equation}
B:=\Bigl\{ -\mathbf{i}B_m^l \Bigr\}_{(l;m)\in I} \label{B}
\end{equation}
over $\mathcal{G}(I,\mathcal{O},\theta)$ compatible with the 0-connection $\nabla$. We denote the trivial gerbe $\mathcal{G}(I,\mathcal{O},\theta)$ with the flat connection $(\nabla, B)$ by $\mathcal{G}_{\tau }$ :
\begin{equation*}
\mathcal{G}_{\tau }:=\Bigl( \mathcal{G}(I,\mathcal{O},\theta), \ \nabla, \ B \Bigr),
\end{equation*}
and we define the gerby deformed $n$-dimensional complex torus $X_{\mathcal{G}_{\tau}}^n$ by the pair of the $n$-dimensional complex torus $X^n$ and the flat gerbe $\mathcal{G}_{\tau}$ : 
\begin{equation*}
X_{\mathcal{G}_{\tau }}^n:=\Bigl( X^n, \ \mathcal{G}_{\tau} \Bigr).
\end{equation*}

We explain the interpretation for the deformation $X_{\mathcal{G}_{\tau}}^n$ of $X^n$ by the flat gerbe $\mathcal{G}_{\tau}$ from the viewpoint of generalized complex geometry. Similarly as in the case of the symplectic geometry side which is explained in section 2, we can consider a B-field transform of the generalized complex structure $\mathcal{I}_J$ over $X^n$. Actually, in our setting, the 2-form
\begin{equation*}
B_m^l=2\pi dx^t \tau^t dy \ \ \ \ \ ((l;m)\in I)
\end{equation*}
is a B-field itself in this sense, and the corresponding B-field transform $\mathcal{I}_J(B)$ of $\mathcal{I}_J$ (see the local expression (\ref{comp})) is expressed locally as
\begin{align*}
&\mathcal{I}_J(B) \left( \frac{\partial}{\partial x}^t, \frac{\partial}{\partial y}^t, dx^t, dy^t \right) \\
&=\left( \frac{\partial}{\partial x}^t, \frac{\partial}{\partial y}^t, dx^t, dy^t \right) \left( \begin{array}{cccc} I_n & O & O & O \\ O & I_n & O & O \\ O & -\tau ^t & I_n & O \\ \tau & O & O & I_n \end{array} \right) \mathcal{I}_J \left( \begin{array}{cccc} I_n & O & O & O \\ O & I_n & O & O \\ O & \tau ^t & I_n & O \\ -\tau & O & O & I_n \end{array} \right).
\end{align*}
In particular, we see that this B-field transform preserves the complex structure of $X^n$, i.e., $\mathcal{I}_J(B)=\mathcal{I}_J$, if and only if the relation $\tau T=(\tau T)^t$ holds. However, as mentioned in the beginning of this section, we assume the condition $\tau T\not=(\tau T)^t$ when we discuss the deformation of $X^n$, and in fact, the above $\mathcal{I}_J(B)$ with the condition $\tau T\not=(\tau T)^t$ defines a generalized complex structure over $X_{\mathcal{G}_{\tau}}^n$ (see also \cite{Gual}).

Let us calculate the mirror dual generalized complex structure $\check{\mathcal{I}}_J(B)$ of $\mathcal{I}_J(B)$ according to the definition (\ref{mirrordef}). By direct calculations, we see
\begin{align*}
&\check{\mathcal{I}}_J(B) \left( \frac{\partial}{\partial \check{x}}^t, \frac{\partial}{\partial \check{y}}^t, d\check{x}^t, d\check{y}^t \right) \\
&=\left( \frac{\partial}{\partial \check{x}}^t, \frac{\partial}{\partial \check{y}}^t, d\check{x}^t, d\check{y}^t \right) \\
&\hspace{3.5mm} \left( \begin{array}{cccc} I_n & O & O & O \\ O & I_n & O & O \\ B_{\rm mat}^{\vee}\tau -\tau ^t(B_{\rm mat}^{\vee})^t & -B_{\rm mat}^{\vee} & I_n & O \\ (B_{\rm mat}^{\vee})^t & O & O & I_n \end{array} \right) \\
&\hspace{3.5mm} \left( \begin{array}{cccc} O & O & O & -((\omega_{\rm mat}^{\vee})^{-1})^t \\ O & O & (\omega_{\rm mat}^{\vee})^{-1} & (\omega_{\rm mat}^{\vee})^{-1}\tau ^t-\tau ((\omega_{\rm mat}^{\vee})^{-1})^t \\ \omega_{\rm mat}^{\vee}\tau -\tau ^t(\omega_{\rm mat}^{\vee})^t & -\omega_{\rm mat}^{\vee} & O & O \\ (\omega_{\rm mat}^{\vee})^t & O & O & O \end{array} \right) \\
&\hspace{3.5mm} \left( \begin{array}{cccc} I_n & O & O & O \\ O & I_n & O & O \\ -B_{\rm mat}^{\vee}\tau +\tau ^t(B_{\rm mat}^{\vee})^t & B_{\rm mat}^{\vee} & I_n & O \\ -(B_{\rm mat}^{\vee})^t & O & O & I_n \end{array} \right),
\end{align*}
and this fact implies that the mirror partner of $X_{\mathcal{G}_{\tau }}^n$ is given by the $2n$-dimensional real torus $\mathbb{R}^{2n}/\mathbb{Z}^{2n}$ with the complexified symplectic form
\begin{equation*}
\tilde{\omega}_{\tau }^{\vee}:=2\pi \mathbf{i}d\check{x}^t (T^{-1})^t d\check{y}+2\pi \mathbf{i}d\check{x}^t (-(T^{-1})^t\tau )d\check{x}.
\end{equation*}
Here, similarly as in the case of $\check{X}^n$, when we decompose $\tilde{\omega}_{\tau }^{\vee}$ into
\begin{equation*}
\tilde{\omega}_{\tau }^{\vee}=2\pi d\check{x}^t \omega_{\rm mat}^{\vee} d\check{y}-2\pi d\check{x}^t \omega_{\rm mat}^{\vee}\tau d\check{x}-\mathbf{i} \Bigl( 2\pi d\check{x}^t B_{\rm mat}^{\vee} d\check{y}-2\pi d\check{x}^t B_{\rm mat}^{\vee}\tau d\check{x} \Bigr),
\end{equation*}
its real part
\begin{equation}
\omega_{\tau }^{\vee}:=2\pi d\check{x}^t \omega_{\rm mat}^{\vee} d\check{y}-2\pi d\check{x}^t \omega_{\rm mat}^{\vee}\tau d\check{x} \label{sympformtau}
\end{equation}
gives a symplectic form on $\mathbb{R}^{2n}/\mathbb{Z}^{2n}$, and the closed 2-form
\begin{equation}
B_{\tau }^{\vee}:=2\pi d\check{x}^t B_{\rm mat}^{\vee} d\check{y}-2\pi d\check{x}^t B_{\rm mat}^{\vee}\tau d\check{x} \label{bfieldtau}
\end{equation}
is the B-field. In particular, associated to the deformation of $X^n$ by $\mathcal{G}_{\tau }$, in the symplectic geometry side, the symplectic form $\omega^{\vee}$ on $\check{X}^n$ and the B-field $B^{\vee}$ on $\check{X}^n$ are twisted by $-2\pi d\check{x}^t \omega_{\rm mat}^{\vee}\tau d\check{x}$ and $-2\pi d\check{x}^t B_{\rm mat}^{\vee}\tau d\check{x}$, respectively (compare the definitions (\ref{sympform}) and (\ref{bfield}) with the definitions (\ref{sympformtau}) and (\ref{bfieldtau}), respectively). Hereafter, we denote this complexified symplectic torus (the mirror partner of $X_{\mathcal{G}_{\tau }}^n$ which is obtained by the definition (\ref{mirrordef})) by
\begin{equation*}
\check{X}_{\mathcal{G}_{\tau }}^n:=\Bigl( \mathbb{R}^{2n}/\mathbb{Z}^{2n}, \ \tilde{\omega}_{\tau }^{\vee}=2\pi \mathbf{i}d\check{x}^t (T^{-1})^t d\check{y}+2\pi \mathbf{i}d\check{x}^t (-(T^{-1})^t\tau )d\check{x} \Bigr).
\end{equation*}

The above deformation from $\check{X}^n$ to $\check{X}_{\mathcal{G}_{\tau }}^n$ can also be understood as follows. By using
\begin{equation*}
g_{\tau }:=\left( \begin{array}{cc} I_n & O \\ \tau & I_n \end{array} \right)\in SL(2n;\mathbb{Z}),
\end{equation*}
let us define a diffeomorphism 
\begin{equation*}
\varphi_{g_{\tau }} : \check{X}^n \stackrel{\sim}{\rightarrow} \check{X}_{\mathcal{G}_{\tau }}^n
\end{equation*}
by
\begin{equation*}
\varphi_{g_{\tau }}\left( \begin{array}{cc} \check{x} \\ \check{y} \end{array} \right):=g_{\tau } \left( \begin{array}{cc} \check{x} \\ \check{y} \end{array} \right).
\end{equation*}
We can easily check that this diffeomorphism $\varphi_{g_{\tau }}$ is a (complexified) symplectomorphism, namely,
\begin{equation*}
(\varphi_{g_{\tau }})^* \tilde{\omega}_{\tau }^{\vee}=\tilde{\omega}^{\vee}
\end{equation*}
holds. Thus, in other words, we can also interpret that the deformation $X_{\mathcal{G}_{\tau }}^n$ of $X^n$ by the flat gerbe $\mathcal{G}_{\tau }$ in the complex geometry side is induced from the (complexified) symplectomorphism $\varphi_{g_{\tau }} : \check{X}^n \stackrel{\sim}{\rightarrow} \check{X}_{\mathcal{G}_{\tau }}^n$ in the symplectic geometry side (via the homological mirror symmetry). In particular, in the case $n=1$, i.e., the case of elliptic curves, although there are no non-trivial gerby deformations of $X^1$, this fact corresponds to that $\varphi_{g_{\tau }}$ gives the $SL(2;\mathbb{Z})$-action over $\check{X}^1$ for a given arbitrary $\tau \in M(1;\mathbb{Z})=\mathbb{Z}$ (see also section 6 in \cite{slaction}).

\section{Deformations of objects}
In general, the mirror pair $(X^n,\check{X}^n)$ can be regarded as the trivial special Lagrangian torus fibrations $X^n\rightarrow \mathbb{R}^n/\mathbb{Z}^n$ and $\check{X}^n\rightarrow \mathbb{R}^n/\mathbb{Z}^n$ by SYZ construction \cite{SYZ}. Then, we can construct a holomorphic line bundle on $X^n$ from each pair of a Lagrangian section of $\check{X}^n\rightarrow \mathbb{R}^n/\mathbb{Z}^n$ and a unitary local system along it, and this transformation is called SYZ transform (see \cite{leung, A-P}). The purpose of this section is to discuss deformations of such objects over the deformed mirror pair $(X_{\mathcal{G}_{\tau}}^n,\check{X}_{\mathcal{G}_{\tau}}^n)$.

\subsection{Holomorphic line bundles and Lagrangian submanifolds with unitary local systems}
In this subsection, as preparations of main discussions, we recall the relation between Lagrangian submanifolds in $\check{X}^n$ with unitary local systems along them and the corresponding holomorphic line bundles on $X^n$ based on SYZ construction (SYZ transform).

First, we explain the complex geometry side, namely, define a class of holomorphic line bundles $E_{(s,a,q)}\rightarrow X^n$ with integrable connections $\nabla_{(s,a,q)}$, and for simplicity, we sometimes denote a holomorphic line bundle $E_{(s,a,q)}\rightarrow X^n$ with an integrable connection $\nabla_{(s,a,q)}$ by $(E_{(s,a,q)},\nabla_{(s,a,q)})$. In fact, we first construct it as a smooth complex line bundle on $X^n$ with a connection, and discuss the holomorphicity of such a smooth complex line bundle later (Proposition \ref{hol}). Before giving the strict definition of $(E_{(s,a,q)},\nabla_{(s,a,q)})$, we mention the idea of the construction of $E_{(s,a,q)}$ since the notations of transition functions of it are very complicated. Although we will give the details of the symplectic geometry side again later, in general, the Lagrangian submanifold in $\check{X}^n$ corresponding to a holomorphic line bundle $E_{(s,a,q)}\rightarrow X^n$ with an integrable connection $\nabla_{(s,a,q)}$ is expressed locally as
\begin{equation*}
\left\{ \left( \begin{array}{cc} \check{x} \\ s(\check{x}) \end{array} \right)\in \check{X}^n \approx (\mathbb{R}^n/\mathbb{Z}^n)\times (\mathbb{R}^n/\mathbb{Z}^n) \right\},
\end{equation*}
where
\begin{equation*}
s(\check{x}):=\Bigl( s^1(\check{x}),\cdots, s^n(\check{x}) \Bigr)^t
\end{equation*}
and locally defined smooth functions $s^1(\check{x}),\cdots ,s^n(\check{x})$ satisfy the relations
\begin{equation*}
\begin{array}{ccc} s^1(x^1+1,x^2,\cdots ,x^n)=s^1(\check{x})+a_{11}, & \cdots & s^1(x^1,\cdots ,x^{n-1},x^n+1)=s^1(\check{x})+a_{1n}, \\
\vdots & & \vdots \\
s^n(x^1+1,x^2,\cdots ,x^n)=s^n(\check{x})+a_{n1}, & \cdots & s^n(x^1,\cdots ,x^{n-1},x^n+1)=s^n(\check{x})+a_{nn} \end{array}
\end{equation*}
for a matrix
\begin{equation*}
a:=\left( \begin{array}{ccc} a_{11} & \cdots & a_{1n} \\ \vdots & \ddots & \vdots \\ a_{n1} & \cdots & a_{nn} \end{array} \right)\in M(n;\mathbb{Z}).
\end{equation*}
Then, the transition functions of $E_{(s,a,q)}$ are determined by $a\in M(n;\mathbb{Z})$. Below, we give the strict definition of $E_{(s,a,q)}$. Let
\begin{equation*}
\psi ^{l_1 \cdots l_n}_{m_1 \cdots m_n} : O^{l_1 \cdots l_n}_{m_1 \cdots m_n}\rightarrow O^{l_1 \cdots l_n}_{m_1 \cdots m_n} \times \mathbb{C} \hspace{5mm} (l_j, m_k =1,2,3)
\end{equation*}
be a smooth section of $E_{(s,a,q)}|_{O^{l_1 \cdots l_n}_{m_1 \cdots m_n}}$. The transition functions of $E_{(s,a,q)}$ are non-trivial on 
\begin{equation*}
O^{(l_1=3) \cdots l_n}_{m_1 \cdots m_n}\cap O^{(l_1=1) \cdots l_n}_{m_1 \cdots m_n}, \ O^{l_1 (l_2=3) \cdots l_n}_{m_1 \cdots m_n}\cap O^{l_1 (l_2=1) \cdots l_n}_{m_1 \cdots m_n},\cdots, O^{l_1 \cdots (l_n=3)}_{m_1 \cdots m_n}\cap O^{l_1 \cdots (l_n=1)}_{m_1 \cdots m_n},
\end{equation*}
and otherwise are trivial. We define the transition function on $O^{l_1 \cdots (l_j =3) \cdots l_n}_{m_1 \cdots m_n}\cap O^{l_1 \cdots (l_j =1) \cdots l_n}_{m_1 \cdots m_n}$ by
\begin{align*}
&\left.\psi ^{l_1 \cdots (l_j =3) \cdots l_n}_{m_1 \cdots m_n} \right|_{O^{l_1 \cdots (l_j =3) \cdots l_n}_{m_1 \cdots m_n}\cap O^{l_1 \cdots (l_j =1) \cdots l_n}_{m_1 \cdots m_n}}\\
&=\Bigl( e^{2\pi \mathbf{i}a_j^t y} \Bigr) \left.\psi ^{l_1 \cdots (l_j =1) \cdots l_n}_{m_1 \cdots m_n} \right|_{O^{l_1 \cdots (l_j =3) \cdots l_n}_{m_1 \cdots m_n}\cap O^{l_1 \cdots (l_j =1) \cdots l_n}_{m_1 \cdots m_n}}
\end{align*}
for each $j=1,\cdots, n$, where $a_j:=(a_{1j},\cdots, a_{nj})^t\in \mathbb{Z}^n$. In particular, when we regard the transition function on $O^{l_1 \cdots (l_j =3) \cdots l_n}_{m_1 \cdots (m_k =3) \cdots m_n}\cap O^{l_1 \cdots (l_j =1) \cdots l_n}_{m_1 \cdots (m_k=1) \cdots m_n}$ as
\begin{align*}
&\left.\psi ^{l_1 \cdots (l_j =3) \cdots l_n}_{m_1 \cdots (m_k=3) \cdots m_n} \right|_{O^{l_1 \cdots (l_j =3) \cdots l_n}_{m_1 \cdots (m_k =3) \cdots m_n}\cap O^{l_1 \cdots (l_j =1) \cdots l_n}_{m_1 \cdots (m_k=1) \cdots m_n}}\\
&=\left.\psi ^{l_1 \cdots (l_j =3) \cdots l_n}_{m_1 \cdots (m_k=1) \cdots m_n} \right|_{O^{l_1 \cdots (l_j =3) \cdots l_n}_{m_1 \cdots (m_k =3) \cdots m_n}\cap O^{l_1 \cdots (l_j =1) \cdots l_n}_{m_1 \cdots (m_k=1) \cdots m_n}}\\
&=\Bigl( e^{2\pi \mathbf{i}a_j^t y} \Bigr) \left.\psi ^{l_1 \cdots (l_j =1) \cdots l_n}_{m_1 \cdots (m_k=1) \cdots m_n} \right|_{O^{l_1 \cdots (l_j =3) \cdots l_n}_{m_1 \cdots (m_k =3) \cdots m_n}\cap O^{l_1 \cdots (l_j =1) \cdots l_n}_{m_1 \cdots (m_k=1) \cdots m_n}}
\end{align*}
for each pair $(j,k)\in \{1,\cdots, n\}\times \{1,\cdots, n\}$, it is clear that the cocycle condition is satisfied. Moreover, by using a locally defined smooth function
\begin{equation*}
s(x):=\Bigl( s^1(x),\cdots, s^n(x) \Bigr)^t
\end{equation*}
satisfying the relations
\begin{equation}
s(x_1+1,x_2,\cdots, x_n)=s(x)+a_1, \ \cdots, \ s(x_1,\cdots ,x_{n-1},x_n+1)=s(x)+a_n \label{section}
\end{equation}
and a constant vector
\begin{equation*}
q:=(q_1,\cdots, q_n)^t\in \mathbb{R}^n,
\end{equation*}
we locally define a connection $\nabla_{(s,a,q)}$ on $E_{(s,a,q)}$ by
\begin{equation}
\nabla_{(s,a,q)}=d+\omega_{(s,a,q)}, \ \ \ \omega_{(s,a,q)}:=-2\pi \mathbf{i} \Bigl( s(x)^t+q^t T \Bigr)dy. \label{connection}
\end{equation}
Also, for the constant vectors $a_j\in \mathbb{Z}^n$ ($j=1,\cdots,n$), we set
\begin{equation*}
a:=(a_1,\cdots, a_n)\in M(n;\mathbb{Z}).
\end{equation*}
We can easily check that the above $\nabla_{(s,a,q)}$ is compatible with the transition functions of $E_{(s,a,q)}$ since the function $s$ satisfies the relations (\ref{section}). Now, we give the following proposition without its proof since it can be proved similarly as in \cite[Proposition 3.1]{bijection} (see also \cite[Proposition 4.3]{kazushi} and \cite[Proposition 2.1]{exact}), where
\begin{equation*}
\frac{\partial s}{\partial x}(x):=\left( \begin{array}{ccc} \frac{\partial s^1}{\partial x_1}(x) & \cdots & \frac{\partial s^1}{\partial x_n}(x) \\ \vdots & \ddots & \vdots \\ \frac{\partial s^n}{\partial x_1}(x) & \cdots & \frac{\partial s^n}{\partial x_n}(x) \end{array} \right).
\end{equation*}
\begin{proposition} \label{hol}
For a given locally defined smooth function $s$ satisfying the relations $(\ref{section})$ and a given constant vector $q\in \mathbb{R}^n$, the connection $\nabla_{(s,a,q)}$ is an integrable connection on $E_{(s,a,q)}\rightarrow X^n$ if and only if 
\begin{equation*}
\frac{\partial s}{\partial x}(x)T=\left( \frac{\partial s}{\partial x}(x)T \right)^t
\end{equation*}
holds.
\end{proposition}
Although we omit the details here, these holomorphic line bundles with integrable connections $(E_{(s,a,q)},\nabla_{(s,a,q)})$ forms a DG-category $DG_{X^n}$ (see \cite{kajiura, exact} for instance). In general, for any $A_{\infty}$-category $\mathscr{A}$, we can construct the enhanced triangulated category $Tr(\mathscr{A})$ by using Bondal-Kapranov-Kontsevich construction \cite{bondal, Kon}. We expect that the DG-category $DG_{X^n}$ generates the bounded derived category of coherent sheaves $D^b(Coh(X^n))$ over $X^n$ in the sense of Bondal-Kapranov-Kontsevich construction :
\begin{equation*}
Tr(DG_{X^n})\cong D^b(Coh(X^n)).
\end{equation*}
At least, it is known that it split generates $D^b(Coh(X^n))$ when $X^n$ is an abelian variety (cf. \cite{orlov, abouzaid}).

Next, we explain the symplectic geometry side, namely, give the objects in the symplectic geometry side corresponding to holomorphic line bundles with integrable connections $(E_{(s,a,q)},\nabla_{(s,a,q)})$. Let us consider the $n$-dimensional submanifold
\begin{equation*}
L_{(s,a)}:=\left\{ \left( \begin{array}{cc} \check{x} \\ s(\check{x}) \end{array} \right)\in \check{X}^n\approx (\mathbb{R}^n/\mathbb{Z}^n)\times (\mathbb{R}^n/\mathbb{Z}^n) \right\}
\end{equation*}
in $\check{X}^n$ by using a locally defined smooth function
\begin{equation*}
s(\check{x}):=\Bigl( s^1(\check{x}),\cdots, s^n(\check{x}) \Bigr)^t
\end{equation*}
satisfying the relations
\begin{equation}
s(x^1+1,x^2,\cdots, x^n)=s(\check{x})+a_1, \ \cdots, \ s(x^1,\cdots ,x^{n-1},x^n+1)=s(\check{x})+a_n. \label{lagsection}
\end{equation}
As explained in the above, we use the notation $a$ in the sense
\begin{equation*}
a:=(a_1,\cdots, a_n)\in M(n;\mathbb{Z}),
\end{equation*}
and for later convenience, we set
\begin{equation*}
\frac{\partial s}{\partial \check{x}}(\check{x}):=\left( \begin{array}{ccc} \frac{\partial s^1}{\partial x^1}(\check{x}) & \cdots & \frac{\partial s^1}{\partial x^n}(\check{x}) \\ \vdots & \ddots & \vdots \\ \frac{\partial s^n}{\partial x^1}(\check{x}) & \cdots & \frac{\partial s^n}{\partial x^n}(\check{x}) \end{array} \right).
\end{equation*}
We further take the trivial complex line bundle $\mathcal{L}_{(s,a,q)}\rightarrow L_{(s,a)}$ with the flat connection
\begin{equation*}
\nabla_{\mathcal{L}_{(s,a,q)}}:=d-2\pi \mathbf{i} q^t d\check{x}
\end{equation*}
which is defined by a constant vector
\begin{equation*}
q:=(q_1,\cdots, q_n)^t\in \mathbb{R}^n.
\end{equation*}
Then, the following proposition holds. Here, we omit its proof since it can be proved in a similar way which is written in subsection 4.1 in \cite{bijection} (see also section 4 in \cite{kazushi} and section 2 in \cite{exact}).
\begin{proposition} \label{fukob}
For a given locally defined smooth function $s$ satisfying the relations $(\ref{lagsection})$ and a given constant vector $q\in \mathbb{R}^n$, the pair $(L_{(s,a)},\mathcal{L}_{(s,a,q)})$ gives an object of the Fukaya category over $\check{X}^n$ if and only if
\begin{equation*}
\frac{\partial s}{\partial \check{x}}(\check{x})T=\left( \frac{\partial s}{\partial \check{x}}(\check{x})T \right)^t
\end{equation*}
holds.
\end{proposition}
We denote the full subcategory of the Fukaya category over $\check{X}^n$ consisting of objects $(L_{(s,a)},\mathcal{L}_{(s,a,q)})$ satisfying the condition $\frac{\partial s}{\partial \check{x}}(\check{x})T=\left( \frac{\partial s}{\partial \check{x}}(\check{x})T \right)^t$ by $Fuk_{\rm sub}(\check{X}^n)$.

We explain the interpretation for the above discussions from the viewpoint of SYZ construction. We can regard the complexified symplectic torus $\check{X}^n$ as the trivial special Lagrangian torus fibration $\check{\pi } : \check{X}^n\rightarrow \mathbb{R}^n/\mathbb{Z}^n$, where $\check{x}$ is the local coordinate system of the base space $\mathbb{R}^n/\mathbb{Z}^n$ and $\check{y}$ is the local coordinate system of the fiber of $\check{\pi } : \check{X}^n\rightarrow \mathbb{R}^n/\mathbb{Z}^n$. Then, we can regard each affine Lagrangian submanifold $L_{(s,a)}$ in $\check{X}^n$ as the graph of the section $s : \mathbb{R}^n/\mathbb{Z}^n\rightarrow \check{X}^n$ of $\check{\pi } : \check{X}^n\rightarrow \mathbb{R}^n/\mathbb{Z}^n$.

Finally, by comparing Proposition \ref{hol} with Proposition \ref{fukob}, we immediately obtain the following proposition. In particular, this indicates that 
\begin{equation*}
(E_{(s,a,q)},\nabla_{(s,a,q)})\in DG_{X^n}
\end{equation*}
and 
\begin{equation*}
(L_{(s,a)},\mathcal{L}_{(s,a,q)})\in Fuk_{\rm sub}(\check{X})
\end{equation*}
are mirror dual to each other.
\begin{proposition} \label{mirror}
For a given locally defined smooth function $s$ satisfying the relations $(\ref{section})$, $(\ref{lagsection})$ and a given constant vector $q\in \mathbb{R}^n$, the connection $\nabla_{(s,a,q)}$ is an integrable connection on $E_{(s,a,q)}\rightarrow X^n$ if and only if the pair $(L_{(s,a)},\mathcal{L}_{(s,a,q)})$ gives an object of the Fukaya category over $\check{X}^n$.
\end{proposition}

\subsection{Deformations of objects}
In this subsection, we define deformations of objects in the respective categories $DG_{X^n}$ and $Fuk_{\rm sub}(\check{X})$ associated to the deformation from $(X^n,\check{X}^n)$ to $(X_{\mathcal{G}_{\tau }}^n,\check{X}_{\mathcal{G}_{\tau }}^n)$. In particular, we prove the analogue of Proposition \ref{mirror} for these deformed objects, and this result is given in Proposition \ref{mainth}. Actually, Proposition \ref{mainth} can be regarded as a part of the statement of the homological mirror symmetry for $(X_{\mathcal{G}_{\tau }}^n,\check{X}_{\mathcal{G}_{\tau }}^n)$, so such deformations seem to be natural in this sense.

First, we discuss the deformation of smooth complex line bundles $E_{(s,a,q)}\rightarrow X^n$ with connections $\nabla_{(s,a,q)}$. In general, when we consider the deformation $M_{\mathcal{G}}:=(M,\mathcal{G})$ of a given complex manifold $M$ by a gerbe $\mathcal{G}$ with 0 and 1-connections, each smooth complex line bundle $L\rightarrow M$ with a connection $\nabla$ is deformed to the twisted smooth complex line bundle $L^{\mathcal{G}}\rightarrow M_{\mathcal{G}}$ with the twisted connection $\nabla^{\mathcal{G}}$. Below, we explain this fact more precisely. Let $M$ be a complex manifold and we take an open cover $\{ U_i \}_{i\in I}$ of $M$. For each smooth complex line bundle $L\rightarrow M$ with a connection $\{ \nabla_i \}_{i\in I}$ which is defined by a family of transition functions $\{ \varphi_{ij} \}_{i,j\in I}$, $\nabla_i$ and $\varphi_{ij}$ satisfy the relation
\begin{equation*}
\nabla_j-\varphi_{ij}^{-1}\nabla_i \varphi_{ij}=0
\end{equation*}
on $U_i\cap U_j$ for each $i,j\in I$. Here, we take the trivial gerbe $\mathcal{G}$ with a non-trivial 0-connection $\{ \nabla_{ij}=d+\omega_{ij} \}_{i,j\in I}$, and consider the deformation $M_{\mathcal{G}}$ of $M$ by this gerbe $\mathcal{G}$ (we only consider the case of the trivial gerbe with a non-trivial 0-connection since $\mathcal{G}_{\tau }$ is trivial as a gerbe). Then, the family of transition functions $\{ \varphi_{ij} \}_{i,j\in I}$ is preserved under this deformation since $\mathcal{G}$ is trivial, namely, the deformation $L^{\mathcal{G}}\rightarrow M_{\mathcal{G}}$ (as a bundle) of $L\rightarrow M$ by $\mathcal{G}$ is defined by the same family of transition functions $\{ \varphi_{ij}^{\mathcal{G}}:=\varphi_{ij} \}_{i,j\in I}$. On the other hand, the twisted connection $\{ \nabla_i^{\mathcal{G}} \}_{i\in I}$ on the deformed object $L^{\mathcal{G}}\rightarrow M_{\mathcal{G}}$ is defined by the relation
\begin{equation}
\nabla_j^{\mathcal{G}}-\left( \varphi_{ij}^{\mathcal{G}} \right)^{-1} \nabla_i^{\mathcal{G}} \left( \varphi_{ij}^{\mathcal{G}} \right)=\omega_{ij} \label{twconn}
\end{equation}
on $U_i\cap U_j$ for each $i,j\in I$. Now, we turn to the case $(E_{(s,a,q)},\nabla_{(s,a,q)})$, namely, define the twisted smooth complex line bundle $E_{(s,a,q)}^{\mathcal{G}_{\tau }}\rightarrow X_{\mathcal{G}_{\tau }}^n$ with the twisted connection $\nabla_{(s,a,q)}^{\mathcal{G}_{\tau }}$ as the deformation of a given smooth complex line bundle $E_{(s,a,q)}\rightarrow X^n$ with a connection $\nabla_{(s,a,q)}$ by $\mathcal{G}_{\tau }$. Similarly as in the case of $(E_{(s,a,q)},\nabla_{(s,a,q)})$, for simplicity, we sometimes denote such a deformed object by $(E_{(s,a,q)}^{\mathcal{G}_{\tau }},\nabla_{(s,a,q)}^{\mathcal{G}_{\tau }})$. The transition functions of $E_{(s,a,q)}^{\mathcal{G}_{\tau }}$ coincide with the transition functions of $E_{(s,a,q)}$ since $\mathcal{G}(I,\mathcal{O},\theta)$ is trivial. On the other hand, the connection $\nabla_{(s,a,q)}$ on $E_{(s,a,q)}$ is locally deformed to the following twisted connection $\nabla_{(s,a,q)}^{\mathcal{G}_{\tau }}$ (see also the definition (\ref{connection})) :
\begin{equation}
\nabla_{(s,a,q)}^{\mathcal{G}_{\tau }}=d+\omega_{(s,a,q)}^{\mathcal{G}_{\tau }}, \ \ \ \omega_{(s,a,q)}^{\mathcal{G}_{\tau}}:=-2\pi \mathbf{i} \Bigl( s(x)^t+q^tT+x^t\tau^t \Bigr)dy. \label{twistedconnection}
\end{equation}
In fact, we can easily check that this $\nabla_{(s,a,q)}^{\mathcal{G}_{\tau }}$ is compatible with the transition functions of $E_{(s,a,q)}^{\mathcal{G}_{\tau }}$, i.e., the relations of the form (\ref{twconn}) hold. Although the transition functions are preserved under this deformation, we see that the connection 1-form $\omega_{(s,a,q)}$ is twisted by the 1-form
\begin{equation}
-2\pi \mathbf{i}x^t \tau^t dy \label{1form}
\end{equation}
under this deformation by comparing the definition (\ref{connection}) with the definition (\ref{twistedconnection}). We can further obtain the 2-form
\begin{equation}
-2\pi \mathbf{i}dx^t \tau^t dy \label{2form}
\end{equation}
from the 1-form (\ref{1form}) by considering its derivative, and this 2-form (\ref{2form}) is the 1-connection itself over $\mathcal{G}(I,\mathcal{O},\theta)$ compatible with the 0-connection $\nabla$ (see the definition (\ref{B})). Sometimes the 2-form 
\begin{equation*}
2\pi dx^t \tau^t dy
\end{equation*}
in the local expression (\ref{2form}) is also called the B-field. Here, we discuss the holomorphicity of $(E_{(s,a,q)}^{\mathcal{G}_{\tau }},\nabla_{(s,a,q)}^{\mathcal{G}_{\tau}})$. We see that the following proposition holds.
\begin{proposition} \label{twhol}
For a given locally defined smooth function $s$ satisfying the relation $(\ref{section})$ and a given constant vector $q\in \mathbb{R}^n$, the twisted connection $\nabla_{(s,a,q)}^{\mathcal{G}_{\tau}}$ is a twisted integrable connection on $E_{(s,a,q)}^{\mathcal{G}_{\tau }}\rightarrow X_{\mathcal{G}_{\tau }}^n$ if and only if
\begin{equation*}
\frac{\partial s}{\partial x}(x)T=\left( \frac{\partial s}{\partial x}(x)T \right)^t
\end{equation*}
holds.
\end{proposition}
\begin{proof}
In general, for a given smooth complex vector bundle with a connection, such a connection is an integrable connection on the given smooth complex vector bundle if and only if the (0,2)-part of its curvature form vanishes. Hence, although we may consider the condition such that the (0,2)-part of the curvature form of the twisted connection $\nabla_{(s,a,q)}^{\mathcal{G}_{\tau }}$ vanishes, we need to note that the curvature form $\Omega_{(s,a,q)}^{\mathcal{G}_{\tau }}$ of $\nabla_{(s,a,q)}^{\mathcal{G}_{\tau }}$ is expressed locally as 
\begin{equation*}
\Omega_{(s,a,q)}^{\mathcal{G}_{\tau }}:=d\omega_{(s,a,q)}^{\mathcal{G}_{\tau }}-\Bigl( -2\pi \mathbf{i} dx^t\tau^t dy \Bigr)=-2\pi \mathbf{i}dx^t \left( \frac{\partial s}{\partial x}(x) \right)^t dy,
\end{equation*}
where the 2-form 
\begin{equation*}
-2\pi \mathbf{i}dx^t \tau^t dy
\end{equation*}
can be regarded as the 1-connection over $\mathcal{G}(I,\mathcal{O},\theta)$ compatible with the 0-connection $\nabla$ as explained in the above. In particular, this description is natural from the viewpoint of the fact that the transition functions of $E_{(s,a,q)}^{\mathcal{G}_{\tau }}$ coincide with the transition functions of $E_{(s,a,q)}$. By direct calculations, the (0,2)-part of $\Omega_{(s,a,q)}^{\mathcal{G}_{\tau }}$ turns out that
\begin{equation*}
\left( \Omega_{(s,a,q)}^{\mathcal{G}_{\tau}} \right)^{(0,2)}=2\pi \mathbf{i} d\bar{z}^t ((T-\bar{T})^{-1})^t \left( \frac{\partial s}{\partial x}(x)T \right)^t (T-\bar{T})^{-1} d\bar{z},
\end{equation*}
where $d\bar{z}:=(d\bar{z}_1,\cdots, d\bar{z}_n)^t$. Thus, $\left( \Omega_{(s,a,q)}^{\mathcal{G}_{\tau}} \right)^{(0,2)}=0$ is equivalent to that $((T-\bar{T})^{-1})^t \left( \frac{\partial s}{\partial x}(x)T \right)^t (T-\bar{T})^{-1}$ is a symmetric matrix, i.e.,
\begin{equation*}
\frac{\partial s}{\partial x}(x)T=\left( \frac{\partial s}{\partial x}(x)T \right)^t.
\end{equation*}
\end{proof}
Moreover, we can actually consider the deformation $DG_{X_{\mathcal{G}_{\tau }}^n}$ of the DG-category $DG_{X^n}$ (see \cite{kajiura, exact} for instance) associated to the deformation which is described in the above. We give a rough explanation of this description below. Objects are twisted holomorphic line bundles $E_{(s,a,q)}^{\mathcal{G}_{\tau }}\rightarrow X_{\mathcal{G}_{\tau }}^n$ with twisted integrable connections $\nabla_{(s,a,q)}^{\mathcal{G}_{\tau }}$. In general, the space of morphisms in $DG_{X^n}$ is given by
\begin{align}
&\mathrm{Hom}_{DG_{X^n}}((E_{(s,a,q)},\nabla_{(s,a,q)}),(E_{(s',a',q')},\nabla_{(s',a',q')})) \notag \\ &:=\Gamma(E_{(s,a,q)},E_{(s',a',q')})\bigotimes_{C^{\infty}(X^n)}\Omega^{0,*}(X^n), \label{mor}
\end{align}
where $\Gamma(E_{(s,a,q)},E_{(s',a',q')})$ is the space of homomorphisms from $E_{(s,a,q)}$ to $E_{(s',a',q')}$ and $\Omega^{0,*}(X^n)$ is the space of anti-holomorphic differential forms on $X^n$. Therefore, we may define the space of morphisms in $DG_{X_{\mathcal{G}_{\tau }}}$ by replacing $\Gamma(E_{(s,a,q)},E_{(s',a',q')})$ in the definition (\ref{mor}) to the space of homomorphisms as twisted holomorphic line bundles. In particular, the differential structure on the space of morphisms (\ref{mor}) is defined by using the integrable connections $\nabla_{(s,a,q)}$ and $\nabla_{(s',a',q')}$ (the product structure on the space of morphisms (\ref{mor}) is given by the ``standard'' product structure), so we may define the differential structure on the space of morphisms in $DG_{X_{\mathcal{G}_{\tau }}^n}$ by using the twisted integrable connections of objects. On the other hand, for a given algebraic variety $X$ over $\mathbb{C}$, roughly speaking, if $\alpha$ is an element in the Brauer group of $X$, there exists an $\alpha$-twisted vector bundle $\mathcal{E}_{\alpha}\rightarrow X$ and we can consider the sheaf of Azumaya algebras $\mathcal{A}:=\mathrm{End}(\mathcal{E}_{\alpha})$. Then, it is known that there exists an equivalence between the abelian category $Coh(X,\alpha)$ of $\alpha$-twisted sheaves on $X$ and the abelian category $Coh(X,\mathcal{A})$ of right coherent $\mathcal{A}$-modules on $X$ :
\begin{equation}
\mathrm{Hom}(\mathcal{E}_{\alpha},-) : Coh(X,\alpha)\stackrel{\sim}{\rightarrow} Coh(X,\mathcal{A}). \label{morita}
\end{equation}
These facts are explained in \cite{cal} for example (cf. \cite{kap1, kap2}). In this paper, although we do not assume that $X^n$ is an abelian variety, namely, we regard $X^n$ as a general complex torus, it seems that the deformation from $DG_{X^n}$ to $DG_{X_{\mathcal{G}_{\tau}}^n}$ explained in the above can also be interpreted as an analogue of the equivalence (\ref{morita}) :
\begin{equation*}
E_{(0,0,0)}^{\mathcal{G}_{\tau}}\otimes - : DG_{X^n}\stackrel{\sim}{\rightarrow} DG_{X_{\mathcal{G}_{\tau}}^n}.
\end{equation*}

Next, we discuss the deformation $(L_{(s,a)}^{\tau},\mathcal{L}_{(s,a,q)}^{\tau})$ of a given arbitrary object $(L_{(s,a)},\mathcal{L}_{(s,a,q)})\in Fuk_{\rm sub}(\check{X}^n)$. Let us consider the $n$-dimensional submanifold
\begin{equation*}
L_{(s,a)}^{\tau}:=\left\{ \left( \begin{array}{cc} \check{x} \\ s(\check{x})+\tau \check{x} \end{array} \right)\in \check{X}_{\mathcal{G}_{\tau }}^n\approx (\mathbb{R}^n/\mathbb{Z}^n)\times (\mathbb{R}^n/\mathbb{Z}^n) \right\}
\end{equation*}
in $\check{X}_{\mathcal{G}_{\tau }}^n$ by using the locally defined smooth function $s$ satisfying the relations (\ref{lagsection}) and $\tau\in M(n;\mathbb{Z})$. We further take the trivial complex line bundle $\mathcal{L}_{(s,a,q)}^{\tau}\rightarrow L_{(s,a)}^{\tau}$ with the flat connection
\begin{equation*}
\nabla_{\mathcal{L}_{(s,a,q)}^{\tau}}:=d-2\pi \mathbf{i} q^t d\check{x}
\end{equation*}
which is defined by the constant vector
\begin{equation*}
q:=(q_1,\cdots, q_n)^t\in \mathbb{R}^n.
\end{equation*}
Then, we obtain the following proposition as an analogue of Proposition \ref{fukob}.
\begin{proposition} \label{twistedfukob}
For a given locally defined smooth function $s$ satisfying the relations $(\ref{lagsection})$ and a given constant vector $q\in \mathbb{R}^n$, the pair $(L_{(s,a)}^{\tau},\mathcal{L}_{(s,a,q)}^{\tau})$ gives an object of the Fukaya category over $\check{X}_{\mathcal{G}_{\tau}}^n$ if and only if
\begin{equation*}
\frac{\partial s}{\partial \check{x}}(\check{x})T=\left( \frac{\partial s}{\partial \check{x}}(\check{x})T \right)^t
\end{equation*}
holds.
\end{proposition}
\begin{proof}
Below, we denote the curvature form of the flat connection $\nabla_{\mathcal{L}_{(s,a,q)}^{\tau}}$ on $\mathcal{L}_{(s,a,q)}^{\tau}$ by $\Omega_{\mathcal{L}_{(s,a,q)}^{\tau}}$, i.e., $\Omega_{\mathcal{L}_{(s,a,q)}^{\tau}}=0$. According to \cite[Definition 1.1]{Fuk}, we may consider the following conditions :  
\begin{align}
&L_{(s,a)}^{\tau} \ \mathrm{becomes} \ \mathrm{a} \ \mathrm{Lagrangian} \ \mathrm{submanifold} \ \mathrm{in} \ \check{X}_{\mathcal{G}_{\tau}}^n. \label{F1} \\
&\mathcal{L}_{(s,a,q)}^{\tau}\rightarrow L_{(s,a)}^{\tau} \ \mathrm{with} \ \nabla_{\mathcal{L}_{(s,a,q)}^{\tau}} \ \mathrm{satisfies} \ \Omega_{\mathcal{L}_{(s,a,q)}^{\tau}}=-\mathbf{i}B_{\tau}^{\vee}|_{L_{(s,a)}^{\tau}}. \label{F2}
\end{align}
We first consider the condition (\ref{F1}). A basis of the tangent vector space of $L_{(s,a)}^{\tau}$ at $\left( \begin{array}{cc} \check{x} \\ s(\check{x})+\tau \check{x} \end{array} \right)\in \check{X}_{\mathcal{G}_{\tau }}^n$ is given by
\begin{equation*}
\xi_{\tau }^j:=\frac{\partial}{\partial x^j}+\left( \frac{\partial}{\partial \check{y}} \right)^t \left\{ \left( \frac{\partial s}{\partial \check{x}}(\check{x}) \right)_j+\tau _j \right\}, \ \ \ j=1,\cdots, n,
\end{equation*}
where $\tau _j:=(\tau_{1j},\cdots, \tau_{nj})^t\in \mathbb{Z}^n$. Hence, by direct calculations, we see
\begin{align*}
\omega_{\tau}^{\vee}(\xi_{\tau}^j,\xi_{\tau}^k)&=2\pi (\omega_{j1}^{\vee},\cdots, \omega_{jn}^{\vee})\left\{ \left(\frac{\partial s}{\partial \check{x}}(\check{x}) \right)_k+\tau _k \right\}-2\pi (\omega_{k1}^{\vee},\cdots, \omega_{kn}^{\vee})\left\{ \left(\frac{\partial s}{\partial \check{x}}(\check{x}) \right)_j+\tau _j \right\} \\
&\hspace{3.5mm} -2\pi (\omega_{j1}^{\vee},\cdots, \omega_{jn}^{\vee})\tau _k+2\pi (\omega_{k1}^{\vee},\cdots, \omega_{kn}^{\vee})\tau _j \\
&=2\pi (\omega_{j1}^{\vee},\cdots, \omega_{jn}^{\vee})\left(\frac{\partial s}{\partial \check{x}}(\check{x}) \right)_k-2\pi (\omega_{k1}^{\vee},\cdots, \omega_{kn}^{\vee})\left(\frac{\partial s}{\partial \check{x}}(\check{x}) \right)_j
\end{align*}
for each pair $(j,k)\in \{ 1,\cdots, n \}\times \{ 1,\cdots, n \}$. This implies that $L_{(s,a)}^{\tau}$ becomes a Lagrangian submanifold in $\check{X}_{\mathcal{G}_{\tau}}^n$, i.e., $\omega_{\tau}^{\vee}|_{L_{(s,a)}^{\tau}}=0$ holds if and only if
\begin{equation}
\omega_{\rm mat}^{\vee}\frac{\partial s}{\partial \check{x}}(\check{x})=\left( \omega_{\rm mat}^{\vee}\frac{\partial s}{\partial \check{x}}(\check{x}) \right)^t \label{twistedlag}
\end{equation}
holds. Next, we consider the condition (\ref{F2}). We see
\begin{align*}
-\mathbf{i}B_{\tau}^{\vee}|_{L_{(s,a)}^{\tau}}&=-2\pi \mathbf{i}d\check{x}^t B_{\rm mat}^{\vee} d\check{y}+2\pi \mathbf{i}d\check{x}^t B_{\rm mat}^{\vee} \tau d\check{x}|_{L_{(s,a)}^{\tau}} \\
&=-2\pi \mathbf{i}d\check{x}^t B_{\rm mat}^{\vee} \left\{ \frac{\partial s}{\partial \check{x}}(\check{x}) d\check{x}+\tau d\check{x} \right\}+2\pi \mathbf{i}d\check{x}^t B_{\rm mat}^{\vee} \tau d\check{x} \\
&=-2\pi \mathbf{i}d\check{x}^t B_{\rm mat}^{\vee} \frac{\partial s}{\partial \check{x}}(\check{x}) d\check{x} \\
&=0
\end{align*}
since $\Omega_{\mathcal{L}_{(s,a,q)}^{\tau}}=0$, so we have
\begin{equation}
B_{\rm mat}^{\vee}\frac{\partial s}{\partial \check{x}}(\check{x})=\left( B_{\rm mat}^{\vee}\frac{\partial s}{\partial \check{x}}(\check{x}) \right)^t. \label{twistedb}
\end{equation}
Thus, by the definition of $\omega_{\rm mat}^{\vee}$ and $B_{\rm mat}^{\vee}$, we can easily verify that the relations (\ref{twistedlag}) and (\ref{twistedb}) are equivalent to
\begin{equation*}
-(T^{-1})^t \frac{\partial s}{\partial \check{x}}(\check{x})=\left( -(T^{-1})^t \frac{\partial s}{\partial \check{x}}(\check{x}) \right)^t,
\end{equation*}
namely, 
\begin{equation*}
\frac{\partial s}{\partial \check{x}}(\check{x})T=\left( \frac{\partial s}{\partial \check{x}}(\check{x})T \right)^t.
\end{equation*}
This completes the proof.
\end{proof}
We denote the full subcategory of the Fukaya category over $\check{X}_{\mathcal{G}_{\tau}}^n$ consisting of objects $(L_{(s,a)}^{\tau},\mathcal{L}_{(s,a,q)}^{\tau})$ satisfying the condition $\frac{\partial s}{\partial \check{x}}(\check{x})T=\left( \frac{\partial s}{\partial \check{x}}(\check{x})T \right)^t$ by $Fuk_{\rm sub}(\check{X}_{\mathcal{G}_{\tau}}^n)$. Similarly as in the case of $\check{X}^n$, we can regard the complexified symplectic torus $\check{X}_{\mathcal{G}_{\tau}}^n$ as the trivial special Lagrangian torus fibration $\check{\pi}_{\tau} : \check{X}_{\mathcal{G}_{\tau}}^n\rightarrow \mathbb{R}^n/\mathbb{Z}^n$ from the viewpoint of SYZ construction. Then, under the deformation from $\check{X}^n$ to $\check{X}_{\mathcal{G}_{\tau}}^n$, we see that the Lagrangian section $s : \mathbb{R}^n/\mathbb{Z}^n\rightarrow \check{X}^n$ is twisted by the function $\tau \check{x}$ by comparing the definition of $L_{(s,a)}$ with the definition of $L_{(s,a)}^{\tau}$. In fact, we can also interpret this deformation by using the (complexified) symplectomorphism 
\begin{equation*}
\varphi_{g_{\tau}} : \check{X}^n\stackrel{\sim}{\rightarrow} \check{X}_{\mathcal{G}_{\tau}}^n
\end{equation*}
which is given in section 3. Recall that
\begin{equation*}
g_{\tau}:=\left( \begin{array}{cc} I_n & O \\ \tau & I_n \end{array} \right)\in SL(2n;\mathbb{Z})
\end{equation*}
and $\varphi_{g_{\tau}} : \check{X}^n\stackrel{\sim}{\rightarrow} \check{X}_{\mathcal{G}_{\tau}}^n$ is defined by
\begin{equation*}
\varphi_{g_{\tau}}\left( \begin{array}{cc} \check{x} \\ \check{y} \end{array} \right):=g_{\tau}\left( \begin{array}{cc} \check{x} \\ \check{y} \end{array} \right).
\end{equation*}
Then, we can regard each object $(L_{(s,a)}^{\tau},\mathcal{L}_{(s,a,q)}^{\tau})\in Fuk_{\rm sub}(\check{X}_{\mathcal{G}_{\tau}}^n)$ as $(\varphi_{g_{\tau}}(L_{(s,a)}),(\varphi_{g_{\tau}}^{-1})^*\mathcal{L}_{(s,a,q)})$ :
\begin{equation*}
\begin{CD}
\mathcal{L}_{(s,a,q)}^{\tau}\cong (\varphi_{g_{\tau}}^{-1})^*\mathcal{L}_{(s,a,q)} @>>> \mathcal{L}_{(s,a,q)} \\
@VVV  @VVV \\
L_{(s,a)}^{\tau}\cong \varphi_{g_{\tau}}(L_{(s,a)}) @>>\varphi_{g_{\tau}}^{-1}> L_{(s,a)}.
\end{CD}
\end{equation*}

Thus, by comparing Proposition \ref{twhol} with Proposition \ref{twistedfukob}, we immediately obtain the following.
\begin{proposition} \label{mainth}
For a given locally defined smooth function $s$ satisfying the relations $(\ref{section})$, $(\ref{lagsection})$ and a given constant vector $q\in \mathbb{R}^n$, the twisted connection $\nabla_{(s,a,q)}^{\mathcal{G}_{\tau}}$ is a twisted integrable connection on $E_{(s,a,q)}^{\mathcal{G}_{\tau}}\rightarrow X_{\mathcal{G}_{\tau}}^n$ if and only if the pair $(L_{(s,a)}^{\tau},\mathcal{L}_{(s,a,q)}^{\tau})$ gives an object of the Fukaya category over $\check{X}_{\mathcal{G}_{\tau}}^n$.
\end{proposition}
Clearly, Proposition \ref{mainth} indicates that 
\begin{equation*}
(E_{(s,a,q)}^{\mathcal{G}_{\tau}},\nabla_{(s,a,q)}^{\mathcal{G}_{\tau}})\in DG_{X_{\mathcal{G}_{\tau}}^n}
\end{equation*}
and 
\begin{equation*}
(L_{(s,a)}^{\tau},\mathcal{L}_{(s,a,q)}^{\tau})\in Fuk_{\rm sub}(\check{X}_{\mathcal{G}_{\tau}}^n)
\end{equation*}
are mirror dual to each other. Here, recall that the homological mirror symmetry conjecture for $(X_{\mathcal{G}_{\tau}}^n,\check{X}_{\mathcal{G}_{\tau}}^n)$ states the existence of an equivalence
\begin{equation*}
D^b(Coh(X_{\mathcal{G}_{\tau}}^n))\cong Tr(Fuk(\check{X}_{\mathcal{G}_{\tau}}^n))
\end{equation*}
as triangulated categories, where $D^b(Coh(X_{\mathcal{G}_{\tau}}^n))$ is the bounded derived category of twisted coherent sheaves over $X_{\mathcal{G}_{\tau}}^n$. Similarly as in the case of $(X^n,\check{X}^n)$, it is expected that the DG-category $DG_{X_{\mathcal{G}_{\tau}}^n}$ and the $A_{\infty}$-category $Fuk_{\rm sub}(\check{X}_{\mathcal{G}_{\tau}}^n)$ generate $D^b(Coh(X_{\mathcal{G}_{\tau}}^n))$ and $Tr(Fuk(\check{X}_{\mathcal{G}_{\tau}}^n))$, respectively. We believe that these deformed categories $DG_{X_{\mathcal{G}_{\tau}}^n}$ and $Fuk_{\rm sub}(\check{X}_{\mathcal{G}_{\tau}}^n)$ play a central role in a proof of the homological mirror symmetry conjecture for $(X_{\mathcal{G}_{\tau}}^n,\check{X}_{\mathcal{G}_{\tau}}^n)$.

\section{Deformed Hermitian-Yang-Mills connections and special Lagrangian submanifolds}
In Proposition \ref{mainth}, we verified that 
\begin{equation*}
(E_{(s,a,q)}^{\mathcal{G}_{\tau}},\nabla_{(s,a,q)}^{\mathcal{G}_{\tau}})\in DG_{X_{\mathcal{G}_{\tau}}^n}
\end{equation*}
and 
\begin{equation*}
(L_{(s,a)}^{\tau},\mathcal{L}_{(s,a,q)}^{\tau})\in Fuk_{\rm sub}(\check{X}_{\mathcal{G}_{\tau}}^n)
\end{equation*}
are mirror dual to each other. In this section, for given mirror dual objects $(E_{(s,a,q)}^{\mathcal{G}_{\tau}},\nabla_{(s,a,q)}^{\mathcal{G}_{\tau}})\in DG_{X_{\mathcal{G}_{\tau}}^n}$ and $(L_{(s,a)}^{\tau},\mathcal{L}_{(s,a,q)}^{\tau})\in Fuk_{\rm sub}(\check{X}_{\mathcal{G}_{\tau}}^n)$, we further prove the following : the twisted integrable connection $\nabla_{(s,a,q)}^{\mathcal{G}_{\tau}}$ defines a deformed Hermitian-Yang-Mills connection on the twisted holomorphic line bundle $E_{(s,a,q)}^{\mathcal{G}_{\tau}}\rightarrow X_{\mathcal{G}_{\tau}}^n$ if and only if the Lagrangian submanifold $L_{(s,a)}^{\tau}$ becomes a special Lagrangian submanifold in $\check{X}_{\mathcal{G}_{\tau}}^n$.

In order to observe the correspondence between deformed Hermitian-Yang-Mills connections and special Lagrangian submanifolds, we need to consider K$\ddot{\mathrm{a}}$hler (Calabi-Yau) structures on $(X_{\mathcal{G}_{\tau}}^n,\check{X}_{\mathcal{G}_{\tau}}^n)$, namely, a (complexified) symplectic structure on $X_{\mathcal{G}_{\tau}}^n$ and a complex structure on $\check{X}_{\mathcal{G}_{\tau}}^n$. In particular, we first discuss K$\ddot{\mathrm{a}}$hler (Calabi-Yau) structures on $(X^n,\check{X}^n)$ since we can obtain the deformed mirror pair $(X_{\mathcal{G}_{\tau}}^n,\check{X}_{\mathcal{G}_{\tau}}^n)$ from $(X^n,\check{X}^n)$.

Let us consider a K$\ddot{\mathrm{a}}$hler (Calabi-Yau) structure on $X^n$. Here, we focus on a 2-form $\omega$ on $X^n$ and a metric $g$ on $X^n$ which are locally defined by
\begin{align*}
&\omega:=2\pi dx^t Y dy+2\pi dy^t X^t Y dy, \\
&g:=2\pi dx^t dx+2\pi dx^t X dy+2\pi dy^t X^t dx+2\pi dy^t (X^t X+Y^t Y)dy,
\end{align*}
respectively, where the product in the definition of $\omega$ is the wedge product $\wedge$, and the product in the definition of $g$ is the tensor product $\otimes$. Moreover, the complex structure $J\in \mathrm{End}(\Gamma(TX^n))$ is expressed locally as 
\begin{equation*}
J\left( \frac{\partial}{\partial x}^t, \frac{\partial}{\partial y}^t \right)=\left( \frac{\partial}{\partial x}^t, \frac{\partial}{\partial y}^t \right) \left( \begin{array}{ccc} -XY^{-1} & -Y-XY^{-1}X \\ Y^{-1} & Y^{-1}X \end{array} \right),
\end{equation*}
where $TX^n$ is the tangent bundle on $X^n$, and $\Gamma(TX^n)$ is the space of smooth sections of $TX^n$ (see also the local expression (\ref{comp})). Then, we obtain the following lemma.
\begin{lemma} \label{kahler}
The 2-form $\omega$ and the metric $g$ defines a symplectic form on $X^n$ and a Riemannian metric on $X^n$, respectively, and $J$, $\omega$, $g$ satisfy the relation
\begin{equation}
\omega(V,W)=g(JV,W) \label{compatible}
\end{equation}
for arbitrary $V$, $W\in \Gamma(TX^n)$.
\end{lemma}
\begin{proof}
It is clear that $\omega$ is closed. Moreover, the matrix
\begin{align*}
&\left( \begin{array}{ccc} \omega(\frac{\partial}{\partial x_1}, \frac{\partial}{\partial x_1}) & \cdots & \omega(\frac{\partial}{\partial x_1}, \frac{\partial}{\partial y_n}) \\
\vdots & \ddots & \vdots \\
\omega(\frac{\partial}{\partial y_n}, \frac{\partial}{\partial x_1}) & \cdots & \omega(\frac{\partial}{\partial y_n}, \frac{\partial}{\partial y_n}) \end{array} \right) \\
&=2\pi \left( \begin{array}{cccc} O & Y \\ -Y^t & X^t Y-Y^t X \end{array} \right) \\
&=2\pi \left( \begin{array}{cccc} I_n & O \\ X^t & I_n \end{array} \right) \left( \begin{array}{cccc} O & Y \\ -Y^t & O \end{array} \right) \left( \begin{array}{cccc} I_n & X \\ O & I_n \end{array} \right)
\end{align*}
is regular since $Y$ is positive definite. Hence, we see that $\omega$ is non-degenerate. Similarly, the positive definiteness of the matrix
\begin{align}
&\left( \begin{array}{ccc} g(\frac{\partial}{\partial x_1}, \frac{\partial}{\partial x_1}) & \cdots & g(\frac{\partial}{\partial x_1}, \frac{\partial}{\partial y_n}) \\
\vdots & \ddots & \vdots \\
g(\frac{\partial}{\partial y_n}, \frac{\partial}{\partial x_1}) & \cdots & g(\frac{\partial}{\partial y_n}, \frac{\partial}{\partial y_n}) \end{array} \right) \notag \\
&=2\pi \left( \begin{array}{cccc} I_n & X \\ X^t & X^t X+Y^t Y \end{array} \right) \notag \\
&=2\pi \left( \begin{array}{cccc} I_n & O \\ X^t & I_n \end{array} \right) \left( \begin{array}{cccc} I_n & O \\ O & Y^t Y \end{array} \right) \left( \begin{array}{cccc} I_n & X \\ O & I_n \end{array} \right) \label{g}
\end{align}
follows from the positive definiteness of $Y$, and it is clear that the matrix (\ref{g}) is a symmetric matrix. Finally, we prove the relation (\ref{compatible}). For arbitrary $V$, $W\in \Gamma(TX^n)$, we can locally express them as
\begin{equation*}
V=\frac{\partial}{\partial x}^t a+\frac{\partial}{\partial y}^t b, \ \ \ W=\frac{\partial}{\partial x}^t c+\frac{\partial}{\partial y}^t d
\end{equation*}
by using constant vectors $a:=(a_1,\cdots, a_n)^t$, $b:=(b_1,\cdots, b_n)^t$, $c:=(c_1,\cdots, c_n)^t$, $d:=(d_1,\cdots, d_n)^t\in \mathbb{R}^n$. Then, we have
\begin{align*}
\omega(V, W)&=2\pi (a^t, b^t) \left( \begin{array}{cccc} O & Y \\ -Y^t & X^t Y-Y^t X \end{array} \right) \left( \begin{array}{ccc} c \\ d \end{array} \right), \\
g(JV, W)&=2\pi (a^t, b^t) \left( \begin{array}{cccc} -XY^{-1} & -Y-XY^{-1}X \\ Y^{-1} & Y^{-1}X \end{array} \right)^t \left( \begin{array}{cccc} I_n & X \\ X^t & X^t X+Y^t Y \end{array} \right) \left( \begin{array}{ccc} c \\ d \end{array} \right) \\
&=2\pi (a^t, b^t) \left( \begin{array}{cccc} O & Y \\ -Y^t & X^t Y-Y^t X \end{array} \right) \left( \begin{array}{ccc} c \\ d \end{array} \right)
\end{align*}
by direct calculations, so indeed, the relation (\ref{compatible}) holds.
\end{proof}
We see that $(X^n=\mathbb{C}^n/\mathbb{Z}^n\oplus T\mathbb{Z}^n, \omega,g)$ is a K$\ddot{\mathrm{a}}$hler (Calabi-Yau) manifold by Lemma \ref{kahler}, and hereafter, we denote again 
\begin{equation*}
X^n:=\Bigl( \mathbb{C}^n/\mathbb{Z}^n\oplus T\mathbb{Z}^n, \ \omega \Bigr)
\end{equation*}
in order to simplify the notation (we omit the Riemannian metric $g$ since we will not use it later). Moreover, the generalized complex structure
\begin{equation*}
\mathcal{I}_{\omega} : \Gamma(TX^n\oplus T^*X^n)\rightarrow \Gamma(TX^n\oplus T^*X^n)
\end{equation*}
associated to the symplectic form $\omega$ is expressed locally as
\begin{align*}
&\mathcal{I}_{\omega} \left( \frac{\partial}{\partial x}^t, \frac{\partial}{\partial y}^t, dx^t, dy^t \right) \notag \\
&=\left( \frac{\partial}{\partial x}^t, \frac{\partial}{\partial y}^t, dx^t, dy^t \right) \notag \\
&\hspace{3.5mm} \left( \begin{array}{cccc} O & O & (Y^{-1})^tX^t-XY^{-1} & -(Y^{-1})^t \\ O & O & Y^{-1} & O \\ O & -Y & O & O \\ Y^t & Y^tX-X^tY & O & O \end{array} \right).
\end{align*}
Hence, we can calculate the mirror dual generalized complex structure $\check{\mathcal{I}}_{\omega}$ of $\mathcal{I}_{\omega}$ locally as
\begin{align}
&\check{\mathcal{I}}_{\omega} \left( \frac{\partial}{\partial \check{x}}^t, \frac{\partial}{\partial \check{y}}^t, d\check{x}^t, d\check{y}^t \right) \notag \\
&=\left( \frac{\partial}{\partial \check{x}}^t, \frac{\partial}{\partial \check{y}}^t, d\check{x}^t, d\check{y}^t \right) \notag \\
&\hspace{3.5mm} \left( \begin{array}{cccc} O & -(Y^{-1})^t & (Y^{-1})^tX^t-XY^{-1} & O \\ Y^t & O & O & Y^tX-X^tY \\ O & O & O & -Y \\ O & O & Y^{-1} & O\end{array} \right) \notag \\
&=\left( \frac{\partial}{\partial \check{x}}^t, \frac{\partial}{\partial \check{y}}^t, d\check{x}^t, d\check{y}^t \right) \notag \\
&\hspace{3.5mm} \left( \begin{array}{cccc} I_n & O & O & -X \\ O & I_n & X^t & O \\ O & O & I_n & O \\ O & O & O & I_n \end{array} \right) \left( \begin{array}{cccc} O & -(Y^{-1})^t & O & O \\ Y^t & O & O & O \\ O & O & O & -Y \\ O & O & Y^{-1} & O\end{array} \right) \left( \begin{array}{cccc} I_n & O & O & X \\ O & I_n & -X^t & O \\ O & O & I_n & O \\ O & O & O & I_n \end{array} \right) \label{noncommutative}
\end{align}
according to the definition (\ref{mirrordef}). This result implies that $\check{X}^n\approx \mathbb{R}^{2n}/\mathbb{Z}^{2n}$ has the structure of the $n$-dimensional complex torus whose period matrix is given by $\mathbf{i}(Y^{-1})^t$. In particular, the local expression (\ref{noncommutative}) tells us that $\check{\mathcal{I}}_{\omega}$ is a ``$\beta$-field'' transform $\mathcal{I}_{\mathbf{i}(Y^{-1})^t}(\beta)$ of the generalized complex structure $\mathcal{I}_{\mathbf{i}(Y^{-1})^t}$ over the usual (commutative) $n$-dimensional complex torus $\mathbb{C}^n/\mathbb{Z}^n\oplus \mathbf{i}(Y^{-1})^t \mathbb{Z}^n$. It is clear that this $\beta$-field transform preserves the complex structure of it, i.e., $\check{\mathcal{I}}_{\omega}=\mathcal{I}_{\mathbf{i}(Y^{-1})^t}(\beta)=\mathcal{I}_{\mathbf{i}(Y^{-1})^t}$, if and only if the relation $X^tY=(X^tY)^t$ holds by the local expression (\ref{noncommutative}), and the corresponding $\beta$-field transform with the condition $X^tY\not=(X^tY)^t$ should be regarded as the noncommutative deformation of ``$\theta_2$ type'' in the sense of \cite{kajiura}. We can easily check that $(\mathbb{C}^n/\mathbb{Z}^n\oplus \mathbf{i}(Y^{-1})^t \mathbb{Z}^n,\tilde{\omega}^{\vee})$ is a (noncomutative) K$\ddot{\mathrm{a}}$hler (Calabi-Yau) manifold, and hereafter, we denote again 
\begin{equation*}
\check{X}^n:=\Bigl( \mathbb{C}^n/\mathbb{Z}^n\oplus \mathbf{i}(Y^{-1})^t \mathbb{Z}^n, \ \tilde{\omega}^{\vee} \Bigr)
\end{equation*}
in order to simplify the notation.

Let us consider a K$\ddot{\mathrm{a}}$hler (Calabi-Yau) structure on $X_{\mathcal{G}_{\tau}}^n$. Recall that the B-field transform $\mathcal{I}_J(B)$ of $\mathcal{I}_J$ over $X^n$ is expressed locally as
\begin{align*}
&\mathcal{I}_J(B) \left( \frac{\partial}{\partial x}^t, \frac{\partial}{\partial y}^t, dx^t, dy^t \right) \\
&=\left( \frac{\partial}{\partial x}^t, \frac{\partial}{\partial y}^t, dx^t, dy^t \right) \left( \begin{array}{cccc} I_n & O & O & O \\ O & I_n & O & O \\ O & -\tau ^t & I_n & O \\ \tau & O & O & I_n \end{array} \right) \mathcal{I}_J \left( \begin{array}{cccc} I_n & O & O & O \\ O & I_n & O & O \\ O & \tau ^t & I_n & O \\ -\tau & O & O & I_n \end{array} \right),
\end{align*}
so the following B-field transform $\mathcal{I}_{\omega}(B)$ of $\mathcal{I}_{\omega}$ over $X^n$ also defines a generalized complex structure over $X_{\mathcal{G}_{\tau}}^n$ :
\begin{align*}
&\mathcal{I}_{\omega}(B) \left( \frac{\partial}{\partial x}^t, \frac{\partial}{\partial y}^t, dx^t, dy^t \right) \\
&=\left( \frac{\partial}{\partial x}^t, \frac{\partial}{\partial y}^t, dx^t, dy^t \right) \left( \begin{array}{cccc} I_n & O & O & O \\ O & I_n & O & O \\ O & -\tau ^t & I_n & O \\ \tau & O & O & I_n \end{array} \right) \mathcal{I}_{\omega} \left( \begin{array}{cccc} I_n & O & O & O \\ O & I_n & O & O \\ O & \tau ^t & I_n & O \\ -\tau & O & O & I_n \end{array} \right).
\end{align*}
This description implies that the complexified symplectic form
\begin{equation*}
\tilde{\omega}_{\tau}:=\omega-\mathbf{i}\Bigl( 2\pi dx^t \tau^t dy \Bigr)=2\pi dx^t Y dy+2\pi dy^t X^tY dy-2\pi \mathbf{i}dx^t \tau^t dy
\end{equation*}
on $X_{\mathcal{G}_{\tau}}^n$ can be obtained by twisting the symplectic form $\omega$ on $X^n$ using the B-field $2\pi dx^t \tau^t dy$. It is easy to check that $(\mathbb{C}^n/\mathbb{Z}^n\oplus T\mathbb{Z}^n,\mathcal{G}_{\tau},\tilde{\omega}_{\tau})$ is a gerby deformed K$\ddot{\mathrm{a}}$hler (Calabi-Yau) manifold, and hereafter, we denote again
\begin{equation*}
X_{\mathcal{G}_{\tau}}^n:=\Bigl( \mathbb{C}^n/\mathbb{Z}^n\oplus T\mathbb{Z}^n, \ \mathcal{G}_{\tau}, \ \tilde{\omega}_{\tau} \Bigr)
\end{equation*}
for simplicity. Moreover, in order to consider the complex structure on $\check{X}_{\mathcal{G}_{\tau}}^n$ according to the definition (\ref{mirrordef}), we assume the condition $\mathrm{det}(-\tau-\mathbf{i}Y^t)\not=0$. If we do not assume this condition, then we need to employ the other one which is proposed in our previous work \cite{kazushi} as the definition of mirror partners. Therefore, for simplicity, we assume the above condition here, and by the definition (\ref{mirrordef}), the mirror dual generalized complex structure $\check{\mathcal{I}}_{\omega}(B)$ of $\mathcal{I}_{\omega}(B)$ is calculated as follows :
\begin{align*}
&\check{\mathcal{I}}_{\omega}(B) \left( \frac{\partial}{\partial \check{x}}^t, \frac{\partial}{\partial \check{y}}^t, d\check{x}^t, d\check{y}^t \right) \\
&=\left( \frac{\partial}{\partial \check{x}}^t, \frac{\partial}{\partial \check{y}}^t, d\check{x}^t, d\check{y}^t \right) \\
&\hspace{3.5mm} \left( \begin{array}{cccc} I_n & O & O & -X \\ O & I_n & X^t & X^t\tau^t-\tau X \\ O & O & I_n & O \\ O & O & O & I_n \end{array} \right) \\
&\hspace{3.5mm} \left( \begin{array}{cccc} (Y^{-1})^t\tau & -(Y^{-1})^t & O & O \\ Y^t+\tau Y^{-1}\tau & -\tau (Y^{-1})^t & O & O \\ O & O & -\tau^t Y^{-1} & -Y-\tau^t Y^{-1}\tau^t \\ O & O & Y^{-1} & Y^{-1}\tau^t \end{array} \right) \\
&\hspace{3.5mm} \left( \begin{array}{cccc} I_n & O & O & X \\ O & I_n & -X^t & \tau X-X^t\tau^t \\ O & O & I_n & O \\ O & O & O & I_n \end{array} \right).
\end{align*}
This result implies that $\check{X}_{\mathcal{G}_{\tau}}^n\approx \mathbb{R}^{2n}/\mathbb{Z}^{2n}$ has the structure of the $n$-dimensional complex torus whose period matrix is given by $(-\tau-\mathbf{i}Y^t)^{-1}$, and it generally becomes a noncommutative $n$-dimensional complex torus similarly as in the case of $\check{X}^n$ (we will not discuss the noncommutativity in this paper). It is easy to check that $(\mathbb{C}^n/\mathbb{Z}^n\oplus (-\tau-\mathbf{i}Y^t)^{-1}\mathbb{Z}^n,\tilde{\omega}_{\tau}^{\vee})$ is a (noncommutative) K$\ddot{\mathrm{a}}$hler (Calabi-Yau) manifold, and hereafter, we denote again
\begin{equation*}
\check{X}_{\mathcal{G}_{\tau}}^n:=\Bigl( \mathbb{C}^n/\mathbb{Z}^n\oplus (-\tau-\mathbf{i}Y^t)^{-1}\mathbb{Z}^n, \ \tilde{\omega}_{\tau}^{\vee} \Bigr)
\end{equation*}
for simplicity. Also, we locally express the complex coordinates $\check{z}:=(z^1,\cdots, z^n)^t$ of $\check{X}_{\mathcal{G}_{\tau}}^n$ as $\check{z}=\check{x}+(-\tau-\mathbf{i}Y^t)^{-1}\check{y}$.

Here, we recall the definitions of deformed Hermitian-Yang-Mills connections \cite{leung, m} (see also \cite{jacob}) and special Lagrangian submanifolds as a preparation of main discussions.
\begin{definition} \label{dhymeq}
Let $(M,\omega)$ be an $n$-dimensional K$\ddot{\mathrm{a}}$hler manifold with a K$\ddot{\mathrm{a}}$hler form $\omega$ and $L\rightarrow M$ be a smooth complex line bundle with a Hermitian metric $h$. Then, if a connection $\nabla$ preserving the metric $h$ on $L$ satisfies the conditions\footnote{Although the equation $\mathrm{Im}e^{\mathbf{i}\theta}(\omega+F)^n=0$ is used in the original paper \cite{leung} instead of $\mathrm{Im}e^{\mathbf{i}\theta}(\omega-F)^n=0$ in Definition \ref{dhymeq}, for example, in \cite{jacob}, Jacob and Yau use the relation $\mathrm{Im}e^{\mathbf{i}\theta}(\omega-F)^n=0$ as the definition of the deformed Hermitian-Yang-Mills equation. In this paper, we employ the relation $\mathrm{Im}e^{\mathbf{i}\theta}(\omega-F)^n=0$ as the definition of the deformed Hermitian-Yang-Mills equation according to \cite{jacob} and so on.}
\begin{equation*}
F^{(0,2)}=0, \ \ \ \mathrm{Im}e^{\mathbf{i}\theta}(\omega-F)^n=0
\end{equation*}
for some $\theta\in \mathbb{R}$, we call the connection $\nabla$ a deformed Hermitian-Yang-Mills connection of phase $\theta$ on $L$. Here, $F$ is the curvature form of the connection $\nabla$, and the second equation is called the deformed Hermitian-Yang-Mills equation.
\end{definition}
\begin{definition}
Let $(M,\Omega)$ be a K$\ddot{\mathrm{a}}$hler manifold with a nowhere vanishing holomorphic top form $\Omega$. Then, if a Lagrangian submanifold $L$ in $M$ satisfies the condition
\begin{equation*}
\mathrm{Im}\bigl( e^{\mathbf{i}\theta}\Omega \bigr)|_L=0
\end{equation*}
for some $\theta\in \mathbb{R}$, we call the Lagrangian submanifold $L$ a special Lagrangian submanifold of phase $\theta$ in $M$.
\end{definition}
Now, we present the following theorem.
\begin{theo} \label{dhymsplag}
The relation
\begin{equation*}
\mathrm{Im}e^{\mathbf{i}\theta}\bigl( \omega-\Omega_{(s,a,q)}^{\mathcal{G}_{\tau}} \bigr)^n=0
\end{equation*}
holds for some $\theta\in \mathbb{R}$ if and only if the relation
\begin{equation*}
\mathrm{Im}\bigl( e^{\mathbf{i}\theta} \check{\Omega} \bigr)|_{L_{(s,a)}^{\tau}}=0
\end{equation*}
holds for the same $\theta\in \mathbb{R}$, where $\check{\Omega}:=dz^1\wedge\cdots \wedge dz^n$ is the nowhere vanishing holomorphic $(n,0)$-form on $\check{X}_{\mathcal{G}_{\tau}}^n$.
\end{theo}
\begin{proof}
Note that the K$\ddot{\mathrm{a}}$hler form on $X_{\mathcal{G}_{\tau}}^n$ is given by
\begin{equation*}
\omega=2\pi dx^t Y dy+2\pi dy^t X^tY dy
\end{equation*}
($\tilde{\omega}_{\tau}$ is a complexified K$\ddot{\mathrm{a}}$hler form on $X_{\mathcal{G}_{\tau}}^n$) and the curvature form $\Omega_{(s,a,q)}^{\mathcal{G}_{\tau}}$ of the twisted integrable connection $\nabla_{(s,a,q)}^{\mathcal{G}_{\tau}}$ is expressed locally as
\begin{equation*}
\Omega_{(s,a,q)}^{\mathcal{G}_{\tau}}=-2\pi \mathbf{i} dx^t \left( \frac{\partial s}{\partial x}(x) \right)^t dy
\end{equation*}
as explained in the proof of Proposition \ref{twhol}. Therefore, we have
\begin{align*}
\omega-\Omega_{(s,a,q)}^{\mathcal{G}_{\tau }}&=2\pi dx^t \left\{ Y+\mathbf{i}\left( \frac{\partial s}{\partial x}(x) \right)^t \right\}dy+2\pi dy^t X^t Y dy \\
&=2\pi \mathbf{i} dx^t \left\{ -\mathbf{i}Y+\left( \frac{\partial s}{\partial x}(x) \right)^t \right\} dy+2\pi dy^t X^t Y dy,
\end{align*}
and this implies
\begin{equation}
\bigl( \omega-\Omega_{(s,a,q)}^{\mathcal{G}_{\tau }} \bigr)^n=(2\pi \mathbf{i})^n n! \mathrm{det}\left\{ -\mathbf{i}Y+\left( \frac{\partial s}{\partial x}(x) \right)^t \right\}dx_1\wedge dy_1\wedge \cdots \wedge dx_n \wedge dy_n. \label{dhymc}
\end{equation}

On the other hand, we can rewrite each vector $\xi_{\tau}^j$ ($j=1,\cdots,n$) in the proof of Proposition \ref{twistedfukob} to
\begin{equation*}
\xi_{\tau}^j=\frac{\partial}{\partial z^j}+\frac{\partial}{\partial \bar{z}^j}+\left( \frac{\partial}{\partial \check{z}}^t (-\tau-\mathbf{i}Y^t)^{-1}+\frac{\partial}{\partial \bar{\check{z}}}^t \overline{(-\tau-\mathbf{i}Y^t)^{-1}} \right) \left\{ \left( \frac{\partial s}{\partial \check{x}}(\check{x}) \right)_j+\tau _j \right\}
\end{equation*}
by using the relation
\begin{equation*}
\left( \frac{\partial}{\partial \check{x}}^t, \frac{\partial}{\partial \check{y}}^t \right)=\left( \frac{\partial}{\partial \check{z}}^t, \frac{\partial}{\partial \bar{\check{z}}}^t \right) \left( \begin{array}{cccccc} I_n & (-\tau-\mathbf{i}Y^t)^{-1} \\ I_n & \overline{(-\tau-\mathbf{i}Y^t)^{-1}} \end{array} \right).
\end{equation*}
Then, we have
\begin{align}
\check{\Omega}(\xi_{\tau}^1,\cdots, \xi_{\tau}^n)&=\mathrm{det}\left( \begin{array}{cccccc} dz^1(\xi^1) & \cdots & dz^1(\xi^n) \\
\vdots & \ddots & \vdots \\
dz^n(\xi^1) & \cdots & dz^n(\xi^n) \end{array} \right) \notag \\
&=\mathrm{det}\left( I_n+(-\tau-\mathbf{i}Y^t)^{-1}\left( \frac{\partial s}{\partial \check{x}}(\check{x})+\tau \right) \right) \notag \\
&=\mathrm{det}(-\tau-\mathbf{i}Y^t)^{-1}\left( -\mathbf{i}Y^t+\frac{\partial s}{\partial \check{x}}(\check{x}) \right). \label{splag}
\end{align}

Thus, we obtain the statement by comparing the relation (\ref{dhymc}) with the relation (\ref{splag}).
\end{proof}
Theorem \ref{dhymsplag} leads ``a gerby deformed version'' of the correspondence between deformed Hermitian-Yang-Mills connections and special Lagrangian submanifolds in the case of tori, namely, for given mirror dual objects $(E_{(s,a,q)}^{\mathcal{G}_{\tau}},\nabla_{(s,a,q)}^{\mathcal{G}_{\tau}})\in DG_{X_{\mathcal{G}_{\tau}}^n}$ and $(L_{(s,a)}^{\tau},\mathcal{L}_{(s,a,q)}^{\tau})\in Fuk_{\rm sub}(\check{X}_{\mathcal{G}_{\tau}}^n)$ in the sense of Proposition \ref{mainth}, we obtain the following conclusion : the twisted integrable connection $\nabla_{(s,a,q)}^{\mathcal{G}_{\tau}}$ defines a deformed Hermitian-Yang-Mills connection on the twisted holomorphic line bundle $E_{(s,a,q)}^{\mathcal{G}_{\tau}}\rightarrow X_{\mathcal{G}_{\tau}}^n$ if and only if the Lagrangian submanifold $L_{(s,a)}^{\tau}$ becomes a special Lagrangian submanifold in $\check{X}_{\mathcal{G}_{\tau}}^n$.
\begin{rem}
In general, to determine a $\theta\in \mathbb{R}$ satisfying the definitions of deformed Hermitian-Yang-Mills connections and special Lagrangian submanifolds is a difficult problem although we do not treat it in this paper. At least, we can take such a $\theta\in \mathbb{R}$ in the case that a given locally defined smooth function $s$ is affine.
\end{rem}

\section*{Acknowledgment}
I am grateful to Hiroshige Kajiura for helpful comments. I also would like to thank Manabu Akaho for telling me the notion of deformed Hermitian-Yang-Mills connections. Finally, I am grateful to the referee for reading this paper carefully. This work is supported by JSPS KAKENHI Grant Number 21H04994.

\end{document}